\documentclass[reqno, 11pt]{amsart} 
\usepackage[utf8]{inputenc}
\usepackage{amssymb}
\usepackage{amsthm}
\usepackage{amsfonts}
\usepackage{amsmath}
\usepackage{hyperref}
\hypersetup{
   colorlinks=true,
   linkcolor=blue,
}
\usepackage{enumerate}
\usepackage{mathtools}
\usepackage{tikz}
\usepackage{mathdots}
\usepackage{yhmath}
\usepackage{cancel}
\usepackage{color}
\usepackage{xcolor}
\usepackage{siunitx}
\usepackage{array}
\usepackage{multirow}
\usepackage{amssymb}
\usepackage{tabularx}
\usepackage{booktabs}
\usepackage{graphicx}
\usepackage[normalem]{ulem}
\graphicspath{ {./simulacoes/} }
\usepackage[normalem]{ulem}

\usetikzlibrary{patterns}
\usetikzlibrary{shapes}
\makeatletter
\def\l@subsection{\@tocline{2}{0pt}{2.5pc}{5pc}{}}
\renewcommand\tocchapter[3]{%
  \indentlabel{\@ifnotempty{#2}{\ignorespaces#2.\quad}}#3%
}
\newcommand\@dotsep{4.5}
\def\@tocline#1#2#3#4#5#6#7{\relax
  \ifnum #1>\c@tocdepth % then omit
  \else
    \par \addpenalty\@secpenalty\addvspace{#2}%
    \begingroup \hyphenpenalty\@M
    \@ifempty{#4}{%
      \@tempdima\csname r@tocindent\number#1\endcsname\relax
    }{%
      \@tempdima#4\relax
    }%
    \parindent\z@ \leftskip#3\relax \advance\leftskip\@tempdima\relax
    \rightskip\@pnumwidth plus1em \parfillskip-\@pnumwidth
    #5\leavevmode\hskip-\@tempdima{#6}\nobreak
    \leaders\hbox{$\m@th\mkern \@dotsep mu\hbox{.}\mkern \@dotsep mu$}\hfill
    \nobreak
    \hbox to\@pnumwidth{\@tocpagenum{#7}}\par
    \nobreak
    \endgroup
  \fi}
\makeatother
\AtBeginDocument{%
\makeatletter
\expandafter\renewcommand\csname r@tocindent0\endcsname{0pt}
\makeatother
}
\def\l@subsection{\@tocline{2}{0pt}{2.5pc}{5pc}{}}

\newtheorem{lemma}{Lemma}[section]

\newtheorem{theorem}{Theorem}
\newtheorem{remark}{Remark}[section]
    
    \numberwithin{theorem}{section}
    \numberwithin{equation}{section}

    \newcommand{\nn}{\nonumber}

    \newcommand{\bP}{\mathbb{P}}
    \newcommand{\bE}{\mathbb{E}}
    
    \newcommand{\bas}{\,\overline {\!\sigma}\,}

    \renewcommand{\epsilon}{\varepsilon}

    \newcommand{\id}{{1 \mskip -5mu {\rm I}}}

    \newcommand{\sgn}{\mathop{\rm sgn}\nolimits}

    \newcommand{\eps}{\varepsilon}
    \newcommand{\e}{\mbox{e}}
\newcommand{\comue}[1]{\marginpar{\tiny \textcolor{red}{\texttt{ME:} #1}}}
\newcommand{\comusb}[1]{\marginpar{\tiny\sl \textcolor{blue}{\texttt{SB:} #1}}}

\title[pattern formation for the stochastic Allen-Cahn equation]{Onset of pattern formation for the stochastic Allen-Cahn equation}

\author{Stella Brassesco$^1$ \and Glauco Valle$^{2,*}$ \and Maria Eulália Vares$^3$}

\thanks{* C\MakeLowercase{orresponding author.} E\MakeLowercase{mail: glauco.valle@im.ufrj.br}} 
\thanks{2. Glauco Valle is partially supported by CNPq grants 307938/2022-0 and 403423/2023-6 and by FAPERJ grant E-26/200.442/2023.} \thanks{3. Maria Eulália Vares is partially supported by CNPq grant 310734/2021-5 and by FAPERJ grant E-26/200.442/2023.}
\subjclass[2020]{60H15}
\keywords{stochastic Allen-Cahn equation, reaction-diffusion models, phase separation}

\address{
\newline
$^1$ Departamento de Matemática, Instituto Venezolano de Investigaciones Científicas, Apartado Postal 20632, Caracas, 1020-A, Venezuela
\newline
e-mail: {\rm \texttt{sbrasses@correo.ivic.gob.ve}}
\newline
\newline
$^{2}$ Instituto de Matemática, Universidade Federal do Rio de Janeiro, Caixa Postal 68530, 21945-970, Rio de Janeiro, Brasil
\newline
e-mail: {\rm \texttt{glauco.valle@im.ufrj.br}}
\newline
\newline
$^{3}$ Instituto de Matemática, Universidade Federal do Rio de Janeiro,  Caixa Postal 68530, 21945-970, Rio de Janeiro, Brasil
\newline
e-mail: {\rm \texttt{eulalia@im.ufrj.br}}
}

\begin{document}

\maketitle
%\comusb{acrecentei "small noise limit", de acordo com a redacao proposta por Eulalia}

\begin{abstract}  We study the behavior of the solution of a stochastic Allen-Cahn equation $\frac{\partial u_\eps }{\partial t}=\frac 12   \frac{\partial^2 u_\eps }{\partial x^2}+ u_\eps -u_\eps^3+\sqrt\eps\, \dot W$, with Dirichlet boundary conditions on a suitably large space interval $[-L_\eps , L_\eps]$, starting from the identically zero function, and where  $\dot W$ is a space-time white noise. Our main goal is the description, in the  small noise limit,  of the onset of the phase separation, with the emergence of spatial regions where $u_\eps$ becomes close $1$ or $-1$. The time scale and the spatial structure are determined by a suitable Gaussian process that appears as the solution of the corresponding linearized equation. This issue has been initially examined  by De Masi et al. [Ann. Probab. {\bf 22}, (1994), 334-371] in the related context of a class of reaction-diffusion models obtained as a superposition of a speeded up stirring process and a spin flip dynamics on $\{-1,1\}^{\mathbb{Z}_\eps}$, where $\mathbb{Z}_\eps=\mathbb{Z}$ modulo $\lfloor\eps^{-1}L_\eps\rfloor$.

    % \ \textrm{ in } [-L_\eps , L_\eps] \times [0,\infty), $  $[-L_\eps , L_\eps] \times [0,\infty)$$ $\eps >0$, fix $L_\eps > 0$ and let $u_\eps(x,t)$, $(x,t) \in [-L_\eps , L_\eps] \times [0,\infty)$, be the solution of the stochastic Allen-Cahn equation with Dirichlet boundary conditions and zero initial condition: 
% \begin{equation}
% \label{acdifeq}
% \left\{
% \begin{array}{l}
% {\displaystyle \frac{\partial u_\eps }{\partial t}=\frac 12   \frac{\partial^2 u_\eps }{\partial x^2}-
%     ( u_\eps^3- u_\eps)+\sqrt\eps\, \dot W} , \ \textrm{ in } [-L_\eps , L_\eps] \times [0,\infty), 
%     \\[8pt]
%     u_\eps(-L_\eps,t)=u_\eps(L_\eps,t)=0 , \ \forall \, t\in [0,\infty),
%     \\[3pt] 
%     u_\eps(x,0)\equiv 0 , \ \forall \, x \in [-L_\eps , L_\eps],
% \end{array}
% \right.
% \end{equation}
% where $\dot W$ is a space-time white noise 
% Onset pattern formation for the stochastic Allen-Cahn equation

\end{abstract}

%\tableofcontents \comue{sugiro remover table of contents.}

\section{Introduction}\label{sec:main}

\subsection{Preliminaries and main result}\label{sec:intro-main}

%concentration difference of phases in an alloy

The reaction-diffusion equation 
$$\frac {\partial u}{\partial t} =M\Delta u -\alpha F'(u)$$ 
was introduced by S.~M. Allen and J.~W. Cahn in \cite{AC} as a proposal  to describe the motion of interfaces in binary metal alloys when there are two possible coexisting phases.
 In this context $u=u(x,t)$ is an order parameter,
 $M$ and $\alpha$ are given positive coefficients, $\Delta$ is the Laplacian, and $F'$ is the derivative of $F$,  a two well even  function with minima at $\pm \eta_c$, which  are the equilibrium order parameters, corresponding to each of the phases of the system. %\comue{indep de AC sugiro trocar $\eta_c$ por $u_c$ ou $\bar{u}_c$, algo com $u$}
 
This equation, also known as scalar  Ginzburg-Landau  equation,  appears in several different contexts, in situations where there are two possible attracting states representing distinct pure phases. For instance, it emerges as a kinetic limit for a class of microscopic stochastic systems that mimic a gas in equilibrium which is suddenly cooled and undergoes a phase transition. This was introduced by A. De Masi, P. Ferrari and J. Lebowitz in \cite{DFL} through the superposition of a fast stirring process with a spin flip dynamics. Since then, several variants and different questions regarding the large time behavior have been investigated within this frame, including a version of the main problem considered in the current paper, as described at the end of this section.

 N. Chafee and E.~F. Infante considered  in \cite {C-I}  a class of equations which includes the Allen-Cahn equation in   the one dimensional 
case, where  the spatial variable lies in  an interval,  with Dirichlet boundary conditions. They interpreted it as a dynamical system of gradient type in a suitable space of continuous  functions of the interval,  and characterized  all the critical points. In the Allen-Cahn case mentioned before, since $F'(0)=0$,  the  function identically zero is a stationary solution  of the equation. They showed that, for any fixed sufficiently large interval, it is  unstable  and that there is a pair of stable states, let us say $\pm \phi$, satisfying $\phi(x)>0$ except at the endpoints of the interval, and which correspond to the pure phases.
%in the Allen-Cahn model. 
The number of stationary solutions is finite for each finite interval, but  increases with the interval size and, apart from  $\pm \phi$, all of them are saddles.  
Thus, the global dynamics is clear in this case, as discussed also in \cite{DH}, where the same approach is adopted. 
%\comusb{Feito, podemos apagar se prefirerem, é para introduzir volume $\to \infty$} \comu{Parece melhor colocar no mesmo parágrafo.}
%\comue {de acordo. reintroduzi  "to infinity". não sei se entendi o comentário de Stella}
%\comusb{O que quiz dizer e que se voces quizerem, podem remover esse comentario, o motivo dele era introduzir volume indo a infinito em evolucao de interfaces. Mas no final, nao coloquei  nada das escalas, para ficar mais curto}
Still in the one dimensional setting, the problem of the slow motion of interfaces when the size of the interval  diverges to infinity was studied by several authors under suitable boundary conditions, see for instance  \cite{CP} and \cite{FH}  and the references therein. 

In the  pioneering paper \cite{FJ}, W. Faris and G. Jona Lasinio considered  the   Allen-Cahn equation in a finite interval when an additive space--time white noise term of small intensity  is imposed. Their main interest therein  was to study the influence of the noise in the vanishing intensity limit. As discussed in that paper, given the phenomenological origin of this  equation (as many other partial differential equations), 
it seems natural to investigate how  the solutions are affected by small stochastic perturbations. These account for possible  neglected terms in the formulation, or for  microscopic fluctuations   in the original systems. The well known scaling properties of the space--time white noise make it a good and convenient candidate for the above mentioned purpose. Another convenient feature is the existence of an invariant measure which is not difficult to describe (at least formally).

%In the  pioneering paper \cite{FJ}, W. Faris and G. Jona Lasinio considered  the   Allen-Cahn equation in a finite interval when an additive space--time white noise term of small  intensity  is added. Their main interest is t study  the influence of the stochastic term in the limit as the intensity of the noise vanishes. As discussed there, given the phenomenological origin of this  equation (and also of many other partial differential equations), 
%it seems natural to investigate how \comusb{De acordo. Posso apagar o paragrafo preto, e ficamos com o vermelho?$\downarrow$ }\comue{troquei their por the já que não me pareceu clara a concordãncia. fiz pequenas alterações e por isso o texto em vermelho acima} {\red the} solutions are affected by small stochastic perturbations. These account for possible  neglected terms in their formulation, or for  microscopic fluctuations   in the original systems. The white space--time noise  has well known scaling properties that make it a good candidate for the purpose of understanding the influence of the noise on the qualitative behaviour of the solutions. 
 %Moreover, it has an invariant measure that is not difficult to describe (at least formally), and that it is also of independent interest. 

In  \cite{FJ} the authors establish large deviations estimates in the frame of Freidlin and Wentzell theory \cite{F-W} to obtain lower and upper estimates for the probability of tunnelling, that is, the  passage of  solutions with initial configuration close to the attractor $\phi$ to a small neighbourhood of the other attractor $-\phi$. The expected time of tunnelling is exponentially large (in the inverse squared intensity of the noise),  as it is also the case for the analogous times when considering small random perturbations of a dynamical system in  finite dimension.
In \cite{COP}, some of this analysis has been extended to volumes that grow as the noise tends to zero. (For further references and connection with metastability, the reader may check \cite[Sec. 5.5]{OV} or \cite[Chapter 12]{BH}.)

%\comue{troquei ''also' por citação de COP} 
As in \cite{COP}, we consider here  the  stochastic Allen-Cahn equation in a finite  but suitably growing space interval, in the limit of small noise. 
%\comu{o ``as above'' parece desnecessário}\comue{feito} 
Our   purpose  is to study the escape from the zero state.  This corresponds to the initial step of the  phenomenon called spinodal decomposition:  the solution, in a relatively fast time scale,  approaches locally the
 values $+\eta_c$ or $-\eta_c$, corresponding to the stable states. 
  In particular, we investigate the precise time needed  for such patterns to appear, and we  also give a description of the spatial random structure of these. This problem, addressing the onset of the phase separation, has already been investigated from a mathematical point of view  for some  of the reaction-diffusion models introduced in \cite{DFL}. We refer to \cite{DPPV,G} for the analogous results and details, and also to \cite{CPPV} where a different potential $F$ was involved and which presents a rather different behavior, usually called transient bimodality. 

For simplicity, we take here $F(u)=\frac{u^4}{4}-\frac{u^2}{2}$, so that its points of minima  are $\pm 1$ with a local maximum at $0$. It is important to stress that the hyperbolicity of the local maximum  plays a crucial role for the type of behavior that we observe, as one can see by a simple comparison with the situation in \cite{CPPV}.

For each $\eps >0$, fix $L_\eps > 0$ and let $u_\eps(x,t)$, $(x,t) \in [-L_\eps , L_\eps] \times [0,\infty)$,  be the solution of the stochastic Allen-Cahn equation with Dirichlet boundary conditions and zero initial condition: 
\begin{equation}
\label{acdifeq}
\left\{
\begin{array}{l}
\displaystyle \frac{\partial u_\eps }{\partial t}=\frac 12 \frac{\partial^2 u_\eps }{\partial x^2} +
    u_\eps- u_\eps^3  +\sqrt\eps\, \dot W , \ \textrm{ in } [-L_\eps , L_\eps] \times [0,\infty), 
    \\[8pt]
    u_\eps(-L_\eps,t)=u_\eps(L_\eps,t)=0 , \ \forall \, t\in [0,\infty),
    \\[3pt] 
    u_\eps(x,0)\equiv 0 , \  \, x \in [-L_\eps , L_\eps],
\end{array}
\right.
\end{equation}
where $\dot W$ is a space-time white noise in a stochastic basis $(\Omega,\mathcal{F},(\mathcal{F}_t)_{t\ge 0},\bP)$ (i.e. a centered Gaussian process formally satisfying $\bE\big(\dot W(x,t)\dot W(x',t')\big)=\delta_{x-x',t-t'}$). 
 As in \cite{FJ}, we adopt here the definition of solution to \eqref{acdifeq} given by Walsh in  \cite{WalshLN}, usually referred to as mild solution. More precisely, $u_\epsilon$ is the unique (in a.s. sense) continuous solution to the integral equation that results from applying the operator $( \frac{\partial}{\partial t}-\frac 12   \frac{\partial^2 }{\partial x^2}-Id\big)^{-1}$ to the equation:   
     \begin{align}
  \label{ieq}
   u_\eps(x,t)&= \sqrt\eps Z_\eps (x,t) -\int_0^t\int_{-L_\eps}^{L_\eps} \e^{t-s}\, h_{L_\eps}(x,y,t-s)\, u_\eps^3(y,s)\, dyds,
  \end{align}
where 
\begin{equation}
\label{defZ}
Z_\eps (x,t) = \int_0^t\int_{-L_\eps}^{L_\eps} \e^{t-s}\,  h_{L_\eps}(x,y,t-s)\,dW_{y,s}
\end{equation}
is a centered Gaussian process with covariance function
\begin{equation}
  \label{covz}
  \bE\big[Z_\eps (x,t) Z_\eps (x',t')\big]=\int_0^{t\land t'}\e^{t+t'-2s}\, h_{L_\eps} (x,x',t+t'-2s)\,ds,
\end{equation}
and for any $L>0$, $h_L$ is the fundamental solution of the heat equation with zero Dirichlet boundary conditions in $[-L,L]$. 
%{\comue {acho que cai bem como p equeno remark. sugeri pequena alteração. outra coisa: não conferi com \cite{FJ} agora, mas lembro que usam mild solution}
%\comusb{De acordo com a redacao em vermelho, mas prefiro "integral equation (1.2)" em vez de mild solution }%{\color{violet} We could as well have written an integral equation equivalent to \eqref{ieq}
%in terms of the operator 
% $( \frac{\partial}{\partial t}-\frac 12   \frac{\partial^2 }{\partial x^2}\big)^{-1}$, 
% keeping the  linear term in the integrand. It is however more convenient for our analysis  to have the explicit time exponential in the kernel. }
 We could of course have written the integral equation  
in terms of the operator 
 $( \frac{\partial}{\partial t}-\frac 12   \frac{\partial^2 }{\partial x^2}\big)^{-1}$, 
keeping the  linear term in the integrand. It is however more convenient for our analysis  to have the explicit exponential in the kernel.
 
Recall that the identically zero function is  an unstable stationary state for the (deterministic) Allen-Cahn equation \cite{C-I} 
\begin{equation}
\label{acdeteq}
\left\{
\begin{array}{l}
{\displaystyle \frac{\partial v_\eps }{\partial t}=\frac 12   \frac{\partial^2 v_\eps }{\partial x^2} +
    v_\eps- v_\eps^3} , \ \textrm{ in } [-L_\eps , L_\eps] \times [0,\infty), 
    \\[8pt]
    v_\eps(-L_\eps,t)=v_\eps(L_\eps,t)=0, \ \forall \, t \in [0,\infty).
\end{array}
\right.
\end{equation}
%\comusb{Sug: at the initial stage em vez de fora while?}
Thus, as $\epsilon \to 0$, $\sqrt \eps Z_\eps$ should be the leading term of $u_\eps$ at the initial stage. For sufficiently large times, depending on $\eps \approx 0$, it is expected that the noise will push the solution of the perturbed equation \eqref{acdifeq} sufficiently away from $0$. Moreover, for large $L_\eps$ one expects to see interfaces between negative and positive phases, away from which the solution should be close to the stable stationary states $+1$ or $-1$. 
%{\green This phenomenon is already investigated in other contexts, as for a class of interacting particle systems whose macroscopic limit is described by Allen-Cahn equation, as in \cite{G,DPPV}. It represents the onset of the phase separation for those reaction-diffusion models.}   

% \comue{aqui ou na introdução. Achei necessário comentar que $L_\eps$ grande, caso contrário o comportamento deve ser mais simples, uma vez que os zeros estarão apenas nos extremos. Acho que a escolha para $L_\eps$ deveria vir antes, já na introdução.}

%\comue{sugiro trocar convenient por suitable} \comusb{De acordo}
Our aim will be to set $L_\eps \to \infty$ as $\eps \to 0$ in a suitable way, determining a time $T_\eps$ around which $u_\eps$ escapes from the vicinity of the null function,  and to describe the typical profile of $u_\eps(\cdot, T_\eps)$ at such approximate escape time. For suitably large values of $L_\eps$ we should have many zeroes in $[-L_\eps,L_\eps]$, and away from the zeroes $u_\eps(\cdot, T_\eps)$ should take values close to $+1$ or $-1$, with probability converging to one as $\eps \to 0$. 
Besides proving this, it is important to be able to say something about the set of zeroes at time $T_\eps$. With that in mind we consider a change of spatial scale so that $x:=r\sqrt{|\ln \eps|}$. %\comue{De acordo com usar $r$ quando na escala dele?}\comu{De acordo. vou ajustar ao longo do texto.} {\color{blue} On this scale the regions containing the zeroes can be described with large proability as intervals containing the discrete set zeroes of suitable smooth Gaussian process on $[-L_\eps/\sqrt{|\ln \eps|},L_\eps/\sqrt{|\ln \eps|}]$.} 
On this scale the excursions away from zero can be approximately described in terms of the corresponding excursions of a smooth  Gaussian process on $[-L_\eps/\sqrt{|\ln \eps|},L_\eps/\sqrt{|\ln \eps|}]$.

Roughly speaking, $(L_\eps,T_\eps)$ is determined from the analysis presented in the next three items:
\begin{enumerate} 
\item[(1)] A straightforward computation shows that the variance of the process $\sqrt{\eps} Z_\eps$ at times $0< t \le \Hat T_\eps :=  \frac12|\ln \eps|$ tends to zero uniformly in the space variable as $\epsilon \to 0$. This is not the case at suitable larger times $\mathbb T_\eps > \frac12|\ln \eps| + \frac{1}{4} \ln(\frac{|\ln \eps|}{2})$. (See the paragraph that follows \eqref{covz0}  and the Remark \ref{mathbbt}.) 
%\comusb{De acordo em colocar a referencia aqui. Coloquei uma definicao violeta na pag 7, nao sei se concordam com fazer assim. Vai precisar um pequeno comentario para incluir o $Z_\epsilon$} 
We start by considering the behavior of $\sqrt\eps Z_\eps$ at time $\hat T_\eps$. Denote 
\begin{equation}
\label{ypsilon}
\qquad \; \; Y_\epsilon(r):=\sqrt \epsilon Z_\epsilon(r \sqrt{|\ln \eps|},\hat T_\eps),\,   r \in \big[-\tilde L_\eps,\tilde L_\eps \,\big],
\end{equation}
where $\tilde L_\eps := L_\eps/\sqrt{|\ln \epsilon|}$.
If $L_\eps$ is of order $|\ln \eps|$ or larger, then  a more refined calculation allows us to show that the rescaled process $Y_\eps$ is,  with high probability uniformly close on compact intervals to a   stationary smooth Gaussian process of order one
%$(\Hat T_\epsilon \pi)^{-\frac 14}Y/2$ 
(precise definition will be given). This allows us to show that for any given $K$,  with high probability
$|Y_\eps (r)|$ is bigger than a term of order $|\ln\eps|^{-\frac 14}$ for $r$ in a large fraction of $[-K,K]$.

% there exists some $\gamma > 0$ such that 
% for any $K>0$ with high probability \comue{there exists $\gamma$? depois usamos $\gamma=1/4$. dar referência para a prova disso}
% %away from the zeroes of $Y$, 
% $|Y_\eps (r)|$ is bigger than a term of order $|\ln\eps|^{-\gamma}$ for $r$ in a large fraction of $[-K,K]$.   
%(we will show a stronger result related to the convergence $\sqrt \eps Z_\eps$ suitably rescaled at time $\Hat T_\eps$ as $\eps \to 0$). 
\item[(2)] An estimate that uses (1) to establish that at time $\hat T_\eps$,  with high probability   $|u_\eps|$ is also at least of order $|\ln\eps|^{-\frac 14}$ in a large fraction of the corresponding intervals. 
%$[-K\sqrt{|\ln \eps|},K\sqrt{|\ln \eps|}]$. 
\item[(3)] Basic monotonicity properties of solutions of the deterministic Allen-Cahn equation which imply local convergence to the stable stationary states within a time of order $\frac 14 \ln| \ln \eps|$, for an initial condition with absolute value of order at least $|\ln \eps|^{-\frac 14}$. 
\end{enumerate}
%From the above we should have 
%$$
%T_\eps = \frac12|\ln \eps| + c \ln |\ln \eps|= %\Hat T_\eps + c \ln| \ln \eps|
%$$
%for some $c= c(\alpha)>0$, and
%$$
%L_\eps>> \sqrt{\Hat T_\epsilon}. 
%$$ 

Based on (1) we will fix $L_\eps = |\ln \eps|$. Moreover, from the analysis in (2) and (3),  our final time will be 
\begin{equation}
\label{tempo-final}
T^b_\eps := \frac12|\ln \eps| + \frac{1}{4} \ln |\ln \eps| + b  
\end{equation}
for some $b>0$ which is related to how close to 1 $|u_\eps|$ will be at time $T^b_\eps$. In particular, $T^b_\epsilon$ is of the same order as $\mathbb T_\eps$. Having fixed these values for $L_\eps$ and $T^b_\eps$,
%So we fix 
% $$
% L_\eps =  |\ln \eps| \, \text{ and } \,T^b_\eps=\frac12|\ln \eps| + \frac{1}{4}  \ln |\ln \eps| + b.
% $$
%\comusb{Acho melhor dizer alguma coisa mais fraca, porque 
%$L_\eps$ nao pode ser arbitrariamente grande\\ $\leftarrow$ Sug: }  
we point out that  other choices of $L_\epsilon$ of order larger than  $|\ln \epsilon|$   would work as well, as can be verified from our proofs.  
%\comue{de acordo que não possa ser enorme. mas nem sei se vale a pena comentar sem dizer algo concreto.Removi algo da parte velha}
 Before  stating our main result, let us set $\mathbf C(I)$ as the spaces of continuous real functions defined on an interval $I \subset \mathbb{R}$ endowed with the topology of uniform convergence on compacts.
 
 %{\blue Define the set of configurations \comue{27/08 esta notação em \eqref{A set} me parece inadequada porque se refere a $\mathcal N_0 \cap [-K,K]$ posterioremente. Seria melhor já definirmos isso aqui e.g. $\mathcal A^K_\epsilon (a)$e a questão das trajetorias do gaussiano na pagina 6 ficaria melhor colocada, na minha opinião.}
%begin{equation}
%\label{A set}
%\mathcal A_\epsilon (a) = \Big\{ \phi\in \mathbf C^\infty \big[ - \tilde L_\eps,\tilde L_\eps \big] \colon  \inf_{r,\tilde r \in \mathcal N_0(\phi)}|r-\tilde r|>a \Big\},
%\end{equation}
%where $\tilde L_\eps := L_\eps/\sqrt{|\ln \epsilon|}$ and $\mathcal{N}_0 (\phi)$ denotes the set of zeroes of the function $\phi$. Also, let
%\begin{equation}
%\label{AKeps}
%\mathcal A^{K,\mathfrak{b}}_\epsilon (a) = \Big\{
%\psi \in \mathbf C \big[ - \tilde L_\eps,\tilde L_\eps \big]\colon \, \inf_{\phi \in \mathcal A_\epsilon (a)} \sup_{|r| \le K} | \psi(r) - \phi(r)| \le \mathfrak{b} \Big\},
%\end{equation}
%\comu{troquei $\vartheta$ por $\mathfrak{b}$ para evitar a confusão com a notação no teorema}
%for $\mathfrak{b} >0$ and $0<K < \tilde L_\eps$, i.e., the tubular neighborhood with radius $\mathfrak{b}$ of $\mathcal A_\epsilon (a)$ for the supremum norm restricted to $[-K,K]$. Below $K$ will be a fixed constant. Since $\epsilon \downarrow 0$,  we will have $K < \tilde L_\eps$ for $\epsilon$ sufficiently small. We also write $d(x,A) = \inf\{|x-y|:y\in A\}$, i.e., the usual distance between $x\in \mathbb{R}$ and $A \subset \mathbb{R}$.} 

\medskip

We now state the main result of this paper.

\medskip

\begin{theorem}\label{thm:main} 
%For every $\mathfrak{b} > 0$ and $K > 0$ 
%\begin{equation}\label{eq:mainY}
%\lim_{a\rightarrow 0} \lim_{\eps \to 0} \bP \big( (8\,\pi |\ln\epsilon|)^{\frac 14} Y_\epsilon \in \mathcal A^{K,\mathfrak{b}}_\epsilon (a)  \big) = 1.
%\end{equation}
%$$
%\lim_{a\rightarrow 0} \lim_{\eps \to 0} \bP \big( Y|_{[-\tilde L_\eps, \tilde L_\eps]} \in \mathcal{A}_\epsilon(a) \big) = 1
%$$ 
Let $Y_\epsilon$  and $T^b_\eps$ be as defined above. Then, 
\text{}
\begin{itemize}
\item[\rm{(a)} ] $(8\,\pi |\ln\epsilon|)^{\frac 14} Y_\eps$ converges in distribution as $\eps \downarrow 0$ to a zero average differentiable Gaussian process $Y = \{Y(r)\colon r\in \mathbb{R}\}$ with covariance function $\e^{-\frac{(r-r')^2}{2}}$.
%        (b) For every $K< \infty$, $\vartheta > 0$, $0<\delta<1$, {\red $\varrho>0$}, there exists $b \in (0,\infty)$ such that 
%\begin{equation}\label{eq:main}
%\liminf_{\epsilon \rightarrow 0}
%\bP\Big( \sup_{r:|r|\le K, |Y_\eps(r)| > \vartheta |\ln\epsilon|^{-\frac 14}}
% |u_\epsilon(r \sqrt{|\ln \epsilon|},T_\eps^b) - \sgn(Y_\eps(r))| \le \varrho \Big) > 1 - \delta,
%\end{equation}
%where as usual $\sgn(y)=+1$ if $y>0$ and $\sgn(y)=-1$ if $y<0$. 
\item[\rm{(b)}]
For every $K< \infty$, $\vartheta > 0$, 
% \begin{equation}\label{eq:main}
% \lim_{b \to \infty} \liminf_{\epsilon \rightarrow 0}
% \bP\Big( \sup_{r:|r|\le K, |Y_\eps(r)| > \vartheta |\ln\epsilon|^{-\frac 14}}
%  \left|u_\epsilon(r \sqrt{|\ln \epsilon|},T_\eps^b) - \sgn(Y_\eps(r))\right| =0 \Big)=1,
% \end{equation}
\begin{equation}\label{eq:main}
 \sup_{r:|r|\le K, |Y_\eps(r)| > \vartheta |\ln\epsilon|^{-\frac 14}}
 \left|u_\epsilon(r \sqrt{|\ln \epsilon|},T_\eps^b) - \sgn(Y_\eps(r))\right|\to 0,
\end{equation}
in probability, as  $\epsilon \to 0, b \to \infty$, where as usual $\sgn(y)=+1$ if $y>0$ and $\sgn(y)=-1$ if $y<0$.
\end{itemize}
\end{theorem}

\begin{remark} Using classical representation theorems, for instance the Corollary in \cite{BD}, by $(a)$ in Theorem \ref{thm:main} we can couple $(8\,\pi |\ln\epsilon|)^{\frac 14} Y_\eps$, $0<\eps<1)$  and $Y$ in a way that the convergence occurs almost surely in $\mathbf C[-K,K]$. This implies that almost surely $\{ r\colon |r|\le K, |Y_\eps(r)| < \vartheta |\ln\epsilon|^{-\frac 14} \}$ is, for  sufficiently small $\epsilon$, contained in a finite union of intervals whose total length can be made arbitrarily small by making $\vartheta$ small. Thus, given any $0<K<\infty$ and $\rho \approx 1$, taking $\vartheta>0$ small,  $(b)$ in Theorem \ref{thm:main} implies that on a fraction $\rho$ of the interval $[-K\sqrt{|\ln \eps|},K\sqrt{|\ln \eps|}]$, $u_\epsilon(\cdot,T^b_\eps)$ is uniformly close to one of the states $1$ or $-1$, with probability that converges to one as $\epsilon \to 0$ and $b \to \infty$.
\end{remark}

%\begin{theorem}\label{thm:main} 
%There exists a smooth Gaussian  process
% $Y = (Y(r))_{r\in \mathbb{R}}$  with covariance $\e^{-\frac{(r-r')^2}{4}}$ such that 
% $$
%\lim_{a\rightarrow 0} \lim_{\eps \to 0} \bP \big( Y|_{[-K, K]} \in \mathcal{A}_\epsilon(a) \big) = 1 \, , \ \forall \, K>0,
%$$ 
%$$
%\lim_{a\rightarrow 0} \lim_{\eps \to 0} \bP \big( Y|_{[-\tilde L_\eps, \tilde L_\eps]} \in \mathcal{A}_\epsilon(a) \big) = 1
%$$ 
%and for every $\vartheta > 0$, $K > 0$ and $\varrho < 1$,
%\begin{equation}\label{eq:main}
%\lim_{\epsilon \rightarrow 0}
%\bP\Big( \inf_{r:|r|\le K, d(r,\mathcal N_0(Y))> \vartheta}
% |u_\epsilon(r \sqrt{|\ln \epsilon|},T_\eps) - sgn(Y)| \le \varrho \Big) = 1,
%\end{equation}
%where as usual $\sgn(y)=+1$ if $y>0$ and $sgn(y)=-1$ if $y<0$. 
%\end{theorem}

\subsection{Outline of the proof}\label{sec:plan}
%\comusb{Sug: Outline of the proof $\to$}

The proof of Theorem \ref{thm:main} will be divided in three stages. The first stage will be called the scaling stage which will be followed in sequence by stages I and  II. Let us briefly describe them:
\begin{enumerate}
\item[- [\!\!] {\bf Scaling Stage}] This initial stage concerns how the Gaussian process in the right hand side of \eqref{ieq} determines the space-time scaling  appearing in Theorem \ref{thm:main}. We obtain lower bounds on $T=T(\epsilon)$ such that $\sup_{|x|\le L_\epsilon, t\le T}\mathbb E(\epsilon Z_\epsilon^2(x,t))\to 0$ with $\epsilon$.  We show that $\sqrt \epsilon Z_\eps$ reaches a non-infinitesimal variance by time $\mathbb{T}_\eps$.
For the process $Y_\eps$ defined in \eqref{ypsilon}, we prove that
$$
(8\,\pi |\ln\epsilon|)^{\frac 14} Y_\epsilon
$$
%with
%$$
%Y_\epsilon(r):=\sqrt \epsilon Z_\epsilon(r \sqrt{|\ln \eps|},\hat T_\eps),
%$$ 
suitably converges to a smooth Gaussian process $(Y(r))_{r\in \mathbb{R}}$ with covariance function $\e^{-\frac{(r-r')^2}{2}}$, thus proving (a) in Theorem \ref{thm:main}. 
% which has the property that for each $K>0$$$\lim_{a\rightarrow 0} \lim_{\eps \to 0} \bP \big( \exists \, \phi \in \mathcal{A}_\epsilon(a) : Y =_K \phi  \big) = 1 \, , \ %=\forall \, K>0,$$ where if $\phi$ and $\Hat \phi$ are real functions, we write $\phi =_K \Hat \phi$, if $\phi(x) = \Hat \phi(x)$ for all $x\in [-K,K]$.From this it is clear that the process $Y$ determines the region away from zeroes in \eqref{eq:main}, it has a fundamental role in our proof. \comu{Talvez aqui o melhor seja verificar a convergência $\frac{\ln \eps^{-1}}{2} + \frac{1}{4} \ln(\frac{\ln\eps^{-1} }{2})$ e depois fazer a correção para obter o resultado no tempo $\hat T_\eps$.}\comue{Faz todo sentido} 
At this stage we also obtain some additional control on $\sqrt \epsilon Z_\eps$ that will be used in stages I and II.
\item[- [\!\!] {\bf Stage I}] This is a  prelude  to the formation of patterns in the solution of \eqref{acdifeq}. 
Our aim will be to show that for every $\vartheta > 0$, for $\hat T_\eps=\frac 12|\ln \eps|$, 
\begin{equation}\label{eq:stageI}
\qquad \quad \liminf_{\eps \to 0} \bP \Big(\inf_{r:|r|\le K, |Y_\eps(r)| > \vartheta |\ln\epsilon|^{-\frac 14}}
 \big|u_\epsilon(r \sqrt{|\ln \epsilon|},\hat T_\eps)\big|> \frac{\vartheta}{2 |\ln \epsilon|^{\frac{1}{4}}} \Big) = 1,
\end{equation} 
see Lemma \ref{lem:stage1}. The proof of \eqref{eq:stageI} will first require a control on the supremum norm of $u_\epsilon$ on a smaller time
%\comusb{Tal vez poderiamos evitar definir $\dot T_\eps$, simplesmente colocar $\frac{\hat  T_\eps }{2} $}\comue{pode ser, acho que fica melhor, mas antes vamos ver a prova do lema onde entra} 
$\hat T_\eps/2$.
%$\dot T_\eps = \frac{|\ln \eps|}{4}$. 
We show that $u_\eps$ is of order $\epsilon^\beta$ at time $\hat T_\eps/2$, for some $\beta > 0$, and this ensures that the evolution of $u_\eps$ is still determined by the evolution of the Gaussian process in the right hand side of \eqref{ieq} in the  time interval $
 [\hat T_\eps/2, \hat T_\eps]$. 
%Then \eqref{eq:stageI} will follow from 
%\begin{equation}\label{eq:stageIZ}
%\qquad \lim_{\eps \to 0} \bP \Big(\inf_{r:|r|\le K, d(r,\mathcal N_0(Y_\eps))> \vartheta} \big| Y_\eps(r) \big|> \frac{1}{|\ln \eps|^{\frac{1}{4}}} \Big) = 1,\end{equation}
%which is obtained from the convergence in the Scaling Stage. 
\item[- [\!\!] {\bf Stage II}] This is the final stage. We shall use \eqref{eq:stageI} to prove \eqref{eq:main}. The proof has two main arguments:
\begin{enumerate}
\item[(1)] By a basic property of solutions of \eqref{acdeteq}, 
if $|v_\epsilon(\cdot,0)| > |\ln \eps|^{-\frac{1}{4}}$, then there exists $b = b(\varrho) < \frac{1}{4}$ such that $|v_\epsilon(\cdot, \frac{1}{4} \ln |\ln \eps| + b)| > \varrho$. Because of this, we have set $T^b_\epsilon = \hat T_\eps + \frac{1}{4} \ln |\ln \eps| + b$. 
\item[(2)] With probability close to one when $\eps$ goes to $0$, $|u_\epsilon(\cdot,T^b_\eps)|$ is uniformly close to $|v_\eps(\cdot, \frac{1}{4} \ln |\ln \eps| + b)|$ in an interval with size of order $\sqrt{|\ln \eps|}$ if $v_\eps(\cdot, 0) = u_\eps(\cdot, \hat T_\eps)$ (see Lemma \ref{lem:stageII}). 
\end{enumerate}
\end{enumerate}

\medskip
The remainder of the paper is organized in Sections \ref{sec:scaling}, \ref{sec:stageI} and \ref{sec:stageII} devoted respectively to the scaling stage, stage I and stage II in the proof of Theorem \ref{thm:main}. 
Some more technical proofs are given in the Appendix.

%For $\theta$ small, conveniently fixed, let us define the set of configurations $\mathcal A(a,b)$:
%\begin{align}\mathcal A&(a,b)=\{\phi:[-L,L]\mapsto [-\frac 32 ,\frac 32], \phi\in C^\infty, \phi(-L)=\phi(L)=0, \\ \nn &\inf_{x,y \in \mathcal N_0(\phi)}|x-y|>a\sqrt{\ln \epsilon^{-1}},\mbox {and} \inf_{x:d\big(x,\mathcal N(\phi)\big)>b\sqrt{\ln \epsilon^{-1}}}| \phi(x)|>\epsilon^{\frac 16 -\theta}\}\end{align}

%Taking into acount \eqref{Y} and \eqref{u-Z}, we expect that, with probability $\sim 1$, $d\big((u_\epsilon(\cdot,T),\mathcal A(a,b)\big)<\epsilon^{\frac h4 +\frac16 -\theta}$  for some $a,b$ small. 

%Then, we would like to show that if $\|u_0-\phi\|<\epsilon^{\frac h4 +\frac16 -\theta}$ for some $\phi\in \mathcal A$, the solution   $u_\epsilon(\cdot,t_\alpha;u_0)$ starting at $u_0$ satisfies, for some $\vartheta$, \[P\big(\inf_{x:d(x,\mathcal N_0(\phi)> \vartheta\sqrt{\ln \epsilon^{-1}}}|u_\epsilon(x,t_\alpha;u_0)|>1-\alpha \big) \approx 1\]

%In that sense, is it true that starting from a configuration in $\mathcal A$ that is flat,  $\approx  \epsilon^{\frac 16 -\theta}$ in an interval of order $a \sqrt{ {\epsilon}^{-1}}$, 
%the above holds at $t_\alpha$  if $|x|<c \sqrt{ {\epsilon}^{-1}}$ for some small $c$ at least, 
%no matter the initial  values on the rest of the interval? 

\medskip

\section{Scaling Stage}\label{sec:scaling}

\subsection{The Gaussian term}

Let  the Gaussian  process $Z$ be
   defined for the spatial variable  $x$ on the whole line instead of  $[-L_\eps,L_\eps]$ by 
   \begin{equation}
   \label{defz0}
  Z(x,t)=\int_0^t\int_{-\infty}^\infty \e^{t-s}\,  h(x,y,t-s)\,dW_{y,s}, 
  \end{equation}
  where $h$ is the heat kernel on the whole line. 
Both $Z(x,t)$ and $Z_\eps(x,t)$  as defined in \eqref{defZ}, are continuous in $(x,t)$, and we may think of them as processes $\{Z(\cdot,t)\colon t>0\}$ and  $\{Z_\eps(\cdot,t)\colon t>0\}$ taking values 
  on $\mathbf C(-\infty,\infty)$ and $\mathbf C(-L_\epsilon,L_\epsilon)$ respectively.
  
  We are going to show that, as $L_\epsilon \to \infty$,   the trajectories of $Z_\eps$ become close to those of $Z$, in a sense that will be made precise. The following computations (straightforward in the case 
  of $Z$) are helpful to understand the spatial and temporal scales involved. 
     \begin{equation}
     \label{covz0}
     \bE \big[ Z^2(x,T) \big] = \int_0^T\frac{\e^{2s}}{\sqrt{4\pi s}} ds \sim \frac{\e^{2T}}{2\sqrt{4\pi T}}\quad \mbox{ as }  T \to \infty.
     \end{equation}
In particular, the Gaussian process $Z(x,t)$ is stationary as a process in the spatial variable $x$, and its variance is an increasing function of $t$. 
We define $\mathbb{T}_\eps$ as the solution to 
$\mathbb E\,\epsilon Z^2(0,t)=1$, and observe that it satisfies 
 $2 \mathbb{T}_\eps -\frac 12 \ln \mathbb{T}_\eps\sim |\ln \eps|$.  
Therefore,
$\mathbb{T}_\eps > \frac{|\ln \eps|}{2}+\frac14 \ln(\frac{|\ln \eps| }{2})$.
%\comusb{Coloquei  uma proposta para esclarecer a def de  $\mathbb{T}_\eps$. Dessa forma a afirmacao sob o  processo escalado que vem depois fica certa. Poderiamos colocar esta definicao antes  }  \comue{acho que OK no lugar que ficou}

 Moreover,  it can be seen by the same arguments used below to prove \eqref{EZ02}  that after a  spatial scaling
 %\marginpar{{\color{violet}{\small  SB: Isto e verdade para $\mathbb{T}_\eps$, mas na verdade vamos usar para tempo menor  }}}
%\marginpar[left text]{{\color{violet}{\tiny 31Jul Tal vez e melhor eliminar (2.3) e o comentario?}}}\comu{também acho que esse trecho aqui atrapalha mais do que ajuda.}\comue{antes estava com $r\sqrt{|\ln \epsilon|}$ unificar notação}
\begin{equation*}
\label{g}
\bE\big[\sqrt \eps \,Z(r\sqrt {2\mathbb{T}_\eps},\mathbb{T}_\eps)\,
     \sqrt \eps Z(r'\sqrt {2\mathbb{T}_\eps},\mathbb{T}_\eps) \big]
     \sim \e^{-\frac{(r-r')^2}{2}}.
\end{equation*} 
That is, the scaled process as a process in $r$ has in the limit a smooth covariance, and so the corresponding limiting Gaussian process has smooth trajectories. 
%In particular, its  zero level set  $\mathcal N$ should be finite over finite intervals.
 
%  Denote $\tau_\eps =\frac 12 \ln\eps^{-1}$. Now, the proposal was the following:
%\begin{enumerate}
%\item To see that (for some kind of sup norm $\|\cdot\|$) \begin{itemize} \item[\rm{a)}] $\|\sqrt{\eps}Z(\cdot\,,\tau_\eps)\|\sim(\ln\eps^{-1})^{-\frac 14}$ \item[\rm{b)}]
% $\|u(\cdot\,,\tau_\eps)-\sqrt{\eps}Z(\cdot\,,\tau_\eps)\|<(\ln\eps^{-1})^{-\beta}$ for some 
%$\beta>\frac 14$
%\end{itemize}
 %\item To show that $\sqrt{\eps}Z(x\sqrt{\tau_\eps} ,\tau_\eps)$ is close to a process with smooth trajectories having  a finite number of zeros (for each $\eps$). 
 %\item To show that the solutions to the equation with initial condition as those in item 2. evolve, in short times ($\sim (\ln \eps^{-1})^\alpha$ , with $\alpha<\frac 12$?) to configurations with values $\sim \pm 1$ ``far from the zeros". Barrier lemma?
 
 %The final result could be something like: for each $\delta$,  
 %\[
 %P\big(\sup_{x:d(x,\mathcal N(g))>\rho} |u(x\sqrt{\tau_\eps},T_\eps)-\sgn g(x)|>\delta\big) \to 0 \quad  \mbox{ as $\eps \to 0$}\]
 %$g$ is the Gaussian process whose covariance is the rhs of \eqref{g}. 
 %\end{enumerate}

Recall that the kernel $h_{L_\eps}=h_{L_\eps}(x,y,t)$ has an interpretation in terms of a standard Brownian motion in $\mathbb R$ with  
$B(0)=x\in [-L_\eps,L_\eps]$:
\begin{equation}\label{eq:hl}
h_{L_\eps}(x,y,t)\, dy=\bP_x\big(B(t)\in dy\,,  \nu_{L_\eps}>t\big), 
\end{equation}
where $\nu_{L_\eps}$ is  the exit time of  $B(\cdot)$ from $[-L_\eps,L_\eps]$. 
In particular, 
  \begin{equation}
 \label{pineq}
 h_{L_\eps}(x,y,t)<h(x,y,t), 
 \end{equation}
and, from \eqref{covz} and  the identity in \eqref {covz0},
\begin{equation}
\label{sigz}
\bE [Z_\eps^2(x,t)] <\int_0^t\frac{\e^{2s}}{\sqrt{4\pi s}}ds.
 \end{equation}
From the reflection principle for Brownian motion, it follows:
 \begin{equation}
\label{rp}
h_{L\eps}(x,y,t)=\frac {1}{\sqrt{2\pi t}} \sum _{j\in \mathbb Z}
\big(\,\e^{-\frac{(x-y-4jL_\eps)^2}{2t}}\,-\,\e^{-\frac{(x+y-4jL_\eps-2L_\eps)^2}{2t}}\big).
\end{equation}
%\comusb{Coloquei aqui esa observacao, fica meio perdida, se acharem um lugar melhor podem mudar a vontade}{\color{violet}}
%\comue{transformei em remark e mudei na citação. Só que notação $O(1)$ significa cota superior. É isso o que queremos?}
%\comusb{Mudei. Nao sei o que acham de colocar numeracao com as sesoes nos remarks.}
\begin{remark}
\label{mathbbt}
It is not difficult to see from \eqref{covz}, \eqref{rp} and the above definition of $\mathbb{T}_\eps$ that  $ \big|\bE [\eps\,Z_\eps^2(0,\mathbb{T}_\eps)]-\bE [\eps\,Z^2(0,\mathbb{T}_\eps)]\big|\to 0 $ as $\epsilon \to 0$. 
\end{remark}
%\comue{de acordo com a modificação. Coloquei número da sec nos remarks.}

The kernel $h_{L_\eps}(x,y,t)$ can also be explicitly computed 
 in  terms of the eigenfunctions of the 
Laplace operator with Dirichlet boundary conditions: 
\begin{equation}
\label{ee}
h_{L_\eps}(x,y,t)=\frac{1}{L_\eps} \sum _{n\ge 1}\e^{-t\frac{\pi^2n^2}{2(2L_\eps)^2}}\,
\sin\big(n\pi(\frac{x+L_\eps}{2L_\eps})\big)\,
\sin\big(n\pi(\frac{y+L_\eps}{2L_\eps})\big).
\end{equation}
Given $T>0$, we estimate 
% \comue{do we want $Z$ or $Z_0$? Acho que $Z$ mas a notação não me parece boa:$Z_\epsilon, Z_0, Z$}
\begin{equation}\label{normsupte}
\|Z_\eps\|_{L_\eps,T} := \sup _{|x| \le L_\eps , \, t\le T}|Z_\eps (x,t)|
\end{equation} 
with the aid of Borell's inequality (see \cite[Section 2.1]{Adler} or \cite[Section 2.4]{AW}): 
for any $\lambda>\bE [\|Z_\eps\|_{L_\eps,T}]$ 
\begin{equation}
\label{Bineq}
\bP\,\big(\|Z_\eps\|_{L_\eps,T} >\lambda\big)
\le 4\,\e^{-\frac 12 (\lambda-\bE \|Z_\eps\|_{L_\eps,T})^2/\bar\sigma^2}, 
\end{equation}
where 
\begin{equation}
\label{bas}
\bas^2=\bas^2(L_\eps,T):=\sup _{|x|\le L_\eps, \, t\le T} \bE \big[ Z_\eps^2(x,t) \big]. 
\end{equation}
As already discussed in Section \ref{sec:main}, $L_\eps = |\ln \eps|$ and $T$ will depend on $\eps$, going to infinity conveniently as $\epsilon$ goes to zero, but it will be convenient to have the following estimates for any   $T \gg 1$. 

Now, $\bas$ is easily estimated from \eqref{sigz}:
\begin{equation}
\label{estsigz}
\bas^2\le K\,\frac{\e^{2T}}{\sqrt T}.
\end{equation}
Here and in what follows $K$ denotes a positive  constant that does not depend 
on $\eps$ or  $T$, and that may change from line to line.
To estimate $\bE [\|Z_\eps\|_{L_\eps,T}]$, we use Dudley's inequality
(see for instance \cite[Corollary 4.15 in Chapter IV]{Adler} or \cite[Section 2.5]{AW})): 
there is $K$ such that 
\begin{equation}
\label{dudineq}
\bE [\|Z_\eps\|_{L_\eps,T}] \le K\, \int_0^\infty (\ln N(\eps,\delta))^{\frac 12} d\delta.
\end{equation} 
For each $\delta>0$, $N(\eps,\delta)$ denotes the minimum number of balls with radius $\delta$ in the metric $d$ 
needed to cover  $R:=[-L_\eps,L_\eps]\times [0,T]\subset \mathbb R\times \mathbb R^+$, where $d$ 
is defined  by:
\[
d\big((x,t),(y,s)\big)= \big(\,\bE \big[Z_\eps(x,t)-Z_\eps(y,s)\big]^2\,\big)^\frac12.
\]
 The function $\ln N(\eps,\delta)$ is known as the metric entropy. 

To estimate  $N(\eps,\delta)$, recall that for $t>s>0$ and $x,y\in [-L,L]$, from 
\eqref{covz} and \eqref{ee} it follows that 
\begin{align}
\label{xdif2}
\bE \big[\big(Z_\eps(x,t)-Z_\eps(y,t)\big)^2\big]&\le K\,\e^{2t}\,|x-y| ,
\\ 
\label{tdif2}
\bE \big[\big(Z_\eps(x,t)-Z_\eps(x,s)\big)^2\big]&\le K\,\e^{2t}\,|t-s|^{1/2}.
\end{align}
The inequalities  follow as straightforward  extensions of \cite[Proposition 4.2]{Walsh}, after taking into account the exponential factor. 
It is clear that $R\subset B_{\bas}(0,0)$, where $B_\alpha (x,t)$ denotes 
the ball of radius $\alpha$ in the metric $d$ centered in $(x,t)$. Thus the upper limit of
  the integral in \eqref{dudineq} is indeed 
$\bas$. From \eqref{xdif2}
and \eqref{tdif2} it is not difficult to estimate
\[N(\eps,\delta)\le K \big( \frac{T\e^{4T}}{\delta^4}\big)\big( \frac{L_\eps\e^{2T}}{\delta^2}\big)
\]
and then, 
\begin{align}
\label{entropy}
\int_0^{\bas}&(\ln N(\eps,\delta))^{\frac 12} d\delta
\le \int^{K\e^{T}/T^{\frac 14}}_0\big(\ln{\frac{KL_\eps T\e^{6T}}{\delta^6}}\big)^{1/2}\,d\delta
\\ \nn
=&(K\,L_\eps T\e^{6T})^{1/6}\int^\infty_{\big(\ln{(KL_\eps T^{5/2})}\big)^{1/2}}\e^{-\frac{u^2}{6}}\frac{u^2}{3}\,du 
 \\ \nn
\le & (K\,L_\eps T\e^{6T})^{1/6}\,2\big(\ln{(KL_\eps T^{5/2})}\big)^{\frac 12} (KL_\eps T^{5/2})^{-\frac 16} 
< \frac{\e^T}{T^{\frac 14}}\,\ln{(L_\eps T^{5/2})}.
\end{align}
Thus,% since $T \gg1$, 
% \comue{ por que $T\gg1$ em vez de $T \geq 1$? Outro detalhe: melhor usar o mesmo ($L_\eps$ ou $L$) em ambos os lados}
\begin{equation}
\label{esup}
\bE \Big[\sup_{|x|\le L_\eps, t\le T}|Z_\eps (x,t)|\Big]\le K\,
  \frac{\e^T}{T^{\frac 14}}\,\ln{(L_\eps T^{5/2})}.
\end{equation} 
Now, for each   $\rho \ge 0$, consider  $\hat T_{\epsilon, \rho}= \frac{|\ln\epsilon|}{2+\rho}$, thus the case $\rho=0$ corresponds to $\hat T_\eps = \frac{|\ln{\epsilon|}}{2}$.  In this particular case, we verify from \eqref{estsigz} and  \eqref{esup} that
$$
  \bas^2(L_\eps,\hat T_{\epsilon,\rho})\le K
  \frac{\eps^{-\frac{2}{2+\rho}}}{|\ln \epsilon|^{1/2}}, 
  $$
and
\begin{equation}\label{Bineq2}
 \bE \Big[\sup_{|x|\le L_\eps, t\le \hat T_{\epsilon, \rho}}|Z_\eps (x,t)|\Big]\le K
 \frac{\eps^{-\frac{1}{2+\rho}}\,\ln(|\ln \epsilon|)}{|\ln \epsilon|^{1/4}}.
\end{equation}
%\comusb{Sug. mudar numeracao , (2.9) em vez de (2.17), e apagar"also"}\comue {de acordo}\comu{Feito} 
From \eqref{Bineq}  we obtain
\begin{equation}
\label{ZTh}
\lim_{\eps \to 0} \bP\,\Big(\sup_{|x|\le L_\eps,\, t\le \frac{|\ln \epsilon|}{2+\rho} }
\sqrt\epsilon\, |Z_\eps (x,t)|>
\frac{\epsilon^{\frac 12-\frac{1}{2+\rho}} (\ln(|\ln \epsilon|))^2}{|\ln \epsilon|^{1/4}}\,
\Big) = 0,
\end{equation}
for $\rho\geq 0$. 
 %and 
% \begin{equation}
% \label{ZTh2}
% \lim_{\eps \to 0} \bP\,\Big(\sup_{|x|\le L_\eps,\, t\le \hat T_\eps }
% \sqrt\epsilon\, |Z_\eps (x,t)|>
% \frac{(\ln(|\ln \epsilon|))^2}{|\ln \epsilon|^{1/4}}\,
% \Big) = 0.
% \end{equation}

\subsection{Scaling limit of the Gaussian term} 

%To simplify the notation, in this section we write $\hat T$ instead of $\hat T_\eps$, and $L$ instead of $L_\epsilon$. 
Compute
\begin{align}
\label{EZ02}
\bE \big[ Z&(r\sqrt {|\ln\epsilon|},\hat T_\eps)\,Z(r'\sqrt {|\ln\epsilon|},\hat T_\eps)\big] 
\\ 
\nn 
&= \int_0^{\hat T_\eps}\e^{2(\hat T_\eps-s)}\int
h(r\sqrt {|\ln\epsilon|},y,\hat T_\eps-s)\,h(r'\sqrt {|\ln\epsilon|},y,\hat T_\eps-s)\,dyds
\\ 
\nn
 & = \int_0^{\hat T_\eps}\e^{2(\hat T_\eps-s)}
h(r\sqrt {|\ln\epsilon|},r'\sqrt {|\ln\epsilon|},2(\hat T_\eps-s))\,ds.
\end{align} 
In other words, the Gaussian  process $X_\epsilon(r):=\sqrt \epsilon Z(r\sqrt{|\ln\epsilon|},\hat T_\eps)$, 
as a process in the spatial variable $r$ has correlation function 
\begin{align}
\label{Xeps}
\bE \big[X_\epsilon(r)\,X_\epsilon(r')\big]
= & \int_0^{\hat T_\eps}\e^{-2s}
\frac{\e^{-\frac{(r-r')^2\,|\ln\epsilon|}{4(\hat T_\eps-s)}}}{\sqrt{4\pi(\hat T_\eps-s)}}\,ds 
\\ \nonumber
 = &  \frac{\sqrt {|\ln\epsilon|}}{2\,\sqrt{2\pi}}\int_0^1\e^{- u |\ln\epsilon|}\,
\frac{\e^{-\frac{(r-r')^2}{2(1-u)}}}{\sqrt{1-u}}\,du\,\sim 
\,\frac{\e^{-\frac{(r-r')^2}{2}}}{\sqrt{|\ln\epsilon|}\,2 \sqrt{2\pi}}.
\end{align} 
The identities follow from the definition of $\hat T_\eps$, and  the change of variables 
$\frac{2s}{|\ln\epsilon|}\mapsto u$, and the asymptotic equivalence follows from the classical  Laplace method as in \cite[Theorem  7.1]{O} after  observing that 
the integral above is of the form
\[
\int_0^1\e^{-\lambda p(u)}\,q(u)\,du\underset{\lambda\to \infty} {\sim}\e^{-\lambda p(0)}\,q(0)
\int_0^\infty\e^{- \lambda \big(up'(0)\big)}\,du. 
\]
 In our case, 
 $p(u)=u$ is continuous and differentiable,  and attains its minumum value at $u_0=0$, and
 $q(u)=\frac{\e^{-\frac{(r-r')^2}{2(1-u)}}}{\sqrt{1-u}}$ is  integrable over 
 $[0,1]$,  continuous  at $u=0$ and  $q(0)\neq 0$. 
Then,
 as $\epsilon \to 0$, 
\[
\int_0^1\e^{-u\,|\ln\epsilon|}\,
\frac{\e^{-\frac{(r-r')^2}{2(1-u)}}}{\sqrt{1-u}}\,du\sim 
\frac{\e^{-\frac{(r-r')^2}{2(1-u)}}}{\sqrt{1-u}}\big|_{u=0}\int_0^\infty\e^{-u|\ln\epsilon|}\,du
=\frac{\e^{-\frac{(r-r')^2}{2}}}{|\ln\epsilon|}.
\]
% \comusb{Acho bom lembrar  que  $Y_\epsilon$ foi definido en \eqref{ypsilon}}\comue{feito}
As explained in the introduction,  to prove (a) in Theorem \ref{thm:main} we would like to show that the scaled process  $(8\,\pi |\ln\epsilon|)^{\frac 14} Y_\eps$, where $Y_\eps$ was defined in \eqref{ypsilon}, is close, over  bounded intervals $[-K,K]$ to a (smooth) zero average 
% \comusb{2/10: coloquei "centred", e
% a covariancia em display para citar depois} 
Gaussian  process $\{Y(r):r\in \mathbb{R}\}$ with covariance function 
\begin{equation}
\label{covY}
 \mathbb E \big(Y(r)\,Y(r')\big)=e^{-\frac{(r-r')^2}{2}}.
\end{equation}   

%More precisely,
%let  $\tilde L=\frac{L}{\sqrt {\hat T}}$,
%we would like to show that for each $K>0$, 
%\begin{equation}
% \label{X-Y}
%\sup_{|r|\le K}|(8\,\pi |\ln\epsilon|)^{\frac 14} Y_\eps(r)-\,Y(r)|\,\to 0 \mbox { as } \epsilon \to 0
% \end{equation}%\comusb{Ago7: Falta ajustar constantes}
With this in mind, we will first estimate the difference  $\big(X_\epsilon(r)-Y_\epsilon(r)\big)$, for $r$ not close to the boundary.
Recall that they are the scaled Gaussian  processes defined on the whole line, and on $[-\tilde L_\eps,\tilde L_\eps]$ respectively. 

% \comusb{Acredito que deve ser $[-L/\sqrt{|\ln \epsilon|},L/\sqrt{|\ln \epsilon|}]$}
% \comue{concordo. Tinhamos chamado de $\tilde{L}_\eps$}
% \comu{concordo. corrigi.}
% \comusb{Troquei b por c}\comue{ valeu}
\begin{lemma}\label{prop:convY}
 \label{x-y}Let $c\in(0,1)$ be given. Then, 
 \begin{equation}
 \label{xeps-yeps}
\lim_{\eps \to 0} \mathbb P\,\Big( \sup_{|r|\le c \tilde L_\eps}\,\big|\hat T_\epsilon^{\frac 14}
 \big(X_\epsilon(r)-Y_\epsilon(r)\big)\big|>
 \epsilon^{(1-c)^2}\Big) = 0. 
 \end{equation} 
 \end{lemma} 

\medskip

 The proof of Lemma \ref{x-y} is given in the Appendix. 
 
 \medskip

%Recalling the definition of $A^{K,\mathfrak{b}}_\epsilon$ from \eqref{AKeps}, we are ready to prove \eqref{eq:mainY} in Theorem \ref{thm:main}.

%\begin{proposition}\label{prop:convY} For every $\mathfrak{b} > 0$ and $K > 0$, it holds that
%$$\lim_{a\rightarrow 0} \lim_{\eps \to 0} \bP \left( (8\,\pi |\ln\epsilon|)^{\frac 14} Y_\epsilon \in \mathcal A^{K,\mathfrak{b}}_\epsilon (a)  \right) = 1.$$ 
%\end{proposition}

\noindent {\it Proof of \rm{(a)} in Theorem \ref{thm:main}.}
From \eqref{Xeps}, the estimate given in Remark \ref{Xcont} of %%\comusb{Sug: colocar properties of the moments of  em vez de representations... }\comu{troquei} 
the Appendix,  standard properties of the moments of Gaussian random variables, and \cite[Corollary 14.9]{K}, we have that $\{(8\,\pi |\ln\epsilon|)^{\frac 14} X_\epsilon: 0<\eps<1\}$ is a tight family in $\mathbf{C}(\mathbb{R})$.   By Lemma \ref{x-y} we also have that $\{ (8\,\pi |\ln\epsilon|)^{\frac 14} Y_\eps: 0<\eps<1\}$ is tight.

%From convergence of  the finite dimensional distributions, $\{(8\,\pi |\ln\epsilon|)^{\frac 14} X_\epsilon: 0<\eps<1\}$ converge  in distribution as $\eps \downarrow 0$ to a differentiable Gaussian process $Y = \{Y(r) : r\in \mathbb{R}\}$ with covariance $\e^{-\frac{(r-r')^2}{2}}$. 
%\comusb{Sug: trocar  a parte do  paragrafo antes de cyan pelo violeta}\comue{de acordo}\comu{troquei}
Also from \eqref{Xeps}, the finite dimensional distributions of $\{(8\,\pi |\ln\epsilon|)^{\frac 14} X_\epsilon: 0<\eps<1\}$ converge as $\eps \downarrow 0$ to those of the zero average differentiable Gaussian process $Y = \{Y(r) : r\in \mathbb{R}\}$ with covariance $\e^{-\frac{(r-r')^2}{2}}$, therefore  convergence in distribution in $\mathbf{C}(\mathbb{R})$ holds. The differentiability follows from classical results on stochastic processes, see for instance Theorem 1 in Section 2 of Chapter V in \cite{GS}. Now Lemma \ref{prop:convY} implies that $\{(8\,\pi |\ln\epsilon|)^{\frac 14} Y_\eps: 0<\eps<1\}$ restricted to the compact interval $[-K,K]$ also converges in distribution to the restriction of $Y$ to the same interval. \hfill $\square$

%Recall the definition of $\mathcal{A}_\epsilon(a)$ from \eqref{A set}. Since $Y$ has differentiable paths that are non constant on every interval, it is straightforward to verify that for every $K>0$
%$$\lim_{a\rightarrow 0} \lim_{\eps \to 0} \bP \big( \exists \, \phi \in \mathcal{A}_\epsilon(a) : Y =_K \phi  \big) = 1 .$$ 
%\comu{estou com dificuldade de justificar a limitação pelo 3/2.} \comusb{Concordo em tirar o 3/2 da definicao de A}\comue{tirei}
%Indeed the covariance specification of $Y$ and its differentiability imply that $\mathcal{N}_0(Y) \cap [-\tilde L_\eps, \tilde L_\eps]$ is a discrete set almost surely and $\inf_{r,\tilde r \in \mathcal{N}_0(Y)} |r-\tilde r|$ is a positive random variable.
%Using classical representation theorems, for instance the Corollary in \cite{BD}, we can couple $(8\,\pi |\ln\epsilon|)^{\frac 14} Y_\eps$, $0<\eps<1$),  and $Y$ in a way that the convergence occurs almost surely in $\mathbf C[-K,K]$. This implies the statement of the proposition.

\section{Stage I}\label{sec:stageI}
To simplify the writing in the computations that follow, we shall denote the   sup norm of  $F:[-L_\eps, L_\eps]\mapsto \mathbb R$ simply by $\|F\|$ 
 and for $ H:[-L_\eps, L_\eps] \times[0,T] \mapsto \mathbb R$,   $\|H\|_T$ refers to $\|H\|_{L_\eps,T}$ as in \eqref{normsupte}. 
 
\smallskip 
 We need to consider the Allen-Cahn equation with initial condition $\psi$  any continuous function on $[-L_\eps,L_\eps]$ satisfying Dirichlet boundary conditions. Extending \eqref{acdifeq}, 
$u_\eps(x,t;\psi)$ will denote the mild solution 
 to the stochastic equation with initial condition $\psi$ and Dirichlet boundary conditions.  It is the solution to the integral equation
  \begin{multline}
  \label{ieqconci}
  u_\eps(x,t;\psi)= 
  \int_{-L_\eps}^{L_\eps}
  \e^{t}\, h_{L_\eps}(x,y,t)\,\psi(y)\,dy
  \\
  +\sqrt\eps Z_\eps (x,t) -\int_0^t\int_{-L_\eps}^{L_\eps} \e^{t-s}\, h_{L_\eps}(x,y,t-s)\, u_\eps^3(y,s;\psi)\, dyds.
  \end{multline}
  We continue to denote by $u_\eps$ (without indication on the initial condition) the solution starting at the null function. In a similar fashion, we denote by 
  $v_\eps(x,t;\psi)$ the solution to the deterministic Allen-Cahn equation with initial condition $\psi$. It satisfies the integral equation 
  \begin{multline}
  \label{ieqv}
  v_\eps(x,t;\psi)= 
  \int_{-L_\eps}^{L_\eps}
  \e^{t}\, h_{L_\eps}(x,y,t)\,\psi(y)\,dy
  \\
  -\int_0^t\int_{-L_\eps}^{L_\eps} \e^{t-s}\, h_{L_\eps}(x,y,t-s)\, v_\eps^3(y,s;\psi)\, dyds.
  \end{multline}

%Let us suppose now that $\delta=1+h$, with 
The next lemma, based on a  comparison result from Buckdahn and Pardoux \cite{BP} is used to estimate
the solution $u_\eps$ of the stochastic Allen-Cahn equation  over times $(0,\hat T_\eps)$.

\begin{lemma}
\label{complemma}
Consider the  functions  $\varphi_1$ and  $\varphi_2$
 $\in C(\mathbb R)$, where   
\[
\varphi_1(u)=u^3\id_{u\ge 0},\quad \varphi_2(u)=u^3\id_{u\le 0}, 
\]
and let  $\bar u_1$ and $\bar u_2$ denote the corresponding mild solutions to the equations
%\begin{equation}
%\label{comp}
%\partial_t u-\frac 12 \partial_{xx}u+\varphi_j(u)-u=\sqrt\epsilon \dot W, \quad j=1,2, 
%\end{equation}
\begin{equation}
\label{comp}
\displaystyle \frac{\partial u }{\partial t}=\frac 12   \frac{\partial^2 u }{\partial x^2}+u -\varphi (u)+\sqrt\epsilon \dot W,
\end{equation}
with $\varphi=\varphi_j$, $j=1,2$, both  with  initial condition $\bar u_0\in C[-L_\eps,L_\eps] $, and Dirichlet boundary conditions.
Then, if $u_\epsilon (\cdot\,; \bar u_0)$ denotes the solution to the stochastic Allen-Cahn equation with initial condition $\bar u_0$, 
\begin{itemize}
\item[\rm (a)] $\bar u_1(x,t)\le u_\eps(x,t; \bar u_0)\le \bar u_2(x,t)$ , a.e. on $[-L_\eps, \,L_\eps ]\times \mathbb R$. 
\item[\rm (b)] $\bar u_1(x,t)\le g_t\, \bar u_0(x)+\sqrt\epsilon Z_\eps(x,t)\le \bar u_2(x,t)$ a.e on
 $[-L_\eps, L_\eps]\times \mathbb R$, 
where 
\[
g f(x,t)=\int^{L_\eps}_{-L_\eps} 
\e^{t}\,h_{L_\eps}(x,y,t)\,f(y) dy .
\] 
\end{itemize}

\end{lemma}
\begin{proof}
(a) follows from \cite[Lemma 3.3]{BP}, after recalling that 
$u\mapsto -u$ is Lipschitz, and that $\varphi_2\le\varphi\le \varphi_1 $, where 
$\varphi(u)=u^3$, and that they are non decreasing functions.

(b) follows as well from the same result, since  $g(\bar u_0)(x,t)+\sqrt\epsilon Z_\eps(x,t)$ is the solution to 
equation \eqref{comp} in the case $\varphi\equiv 0$, and  $\varphi_2\le 0\le \varphi_1 $.  
\end{proof}
The integral equations corresponding to \eqref{comp} with initial condition $\bar u_0$   are
\[
u(x,t)+G\varphi_j(u)(x,t)=g \bar u_0(x,t)+\sqrt\epsilon Z_\eps(x,t):=\upsilon(x,t),\qquad j=1,2, 
\]
%\comusb{Troquei $v$ por $\upsilon$ nesta pag e a siguinte, para evitar confussao com a notacao para a eq deterministica}
 where $G \varphi(u)(x,t)=\int_0^t\int^{L_\eps}_{-L_\eps} 
\e^{t-s}\,h_{L_\eps}(x,y,t-s)\varphi(u(y,s))dyds$. 
But, from Lemma  \ref{complemma}, after recalling that the $\varphi$'s  are non decreasing and that 
the kernel in $G$ is positive, 
\begin{align}
\nn
u_\eps\ge \bar u_1&=\upsilon-G(\varphi_1(\bar u_1 ))\ge \upsilon-G(\varphi_1(\upsilon ))
\\ \nn
u_\eps\le\bar  u_2&=\upsilon-G(\varphi_2(\bar u_2 ))\le \upsilon-G(\varphi_2(\upsilon)).
\end{align}
The common initial condition $\bar u_0$ above is not explicit in the notation. 
From the above and the definition of $\upsilon$,   we conclude that 
\begin{equation}
\label{u0z}
|u_\eps(x,t;\bar u_0)-g \bar u_0(x,t)-\sqrt\epsilon Z_\eps(x,t)|\le G\varphi(|\upsilon|)(x,t).  
\end{equation}
In  the particular case $\bar  u_0=0$, for each $T>0$ we obtain
\begin{equation}\label{eq:rev1}
\|u_\eps-\sqrt\eps Z_\epsilon\|_T\le\| G\big(|\sqrt\eps Z_\eps |^3\big)\|_T
\le \e^T\,\|\sqrt\eps Z_\epsilon\|^3_T.
\end{equation}
To simplify the notation, we consider the set
 $$\Gamma_\eps=\big\{\sqrt\epsilon\, \|Z_\eps\|_{\hat T_{\eps}/2}\le
\frac{\epsilon^{\frac 12-\frac{1}{4}}}{|\ln \epsilon|^{1/4}}\,
(\ln|\ln \epsilon|)^2\big\}.$$ From \eqref{ZTh} with $\rho=2$, 
we see that $\bP(\Gamma_\eps)\to 1$ 
as $\eps \to 0$. Thus, from the above estimation we obtain that, on $\Gamma_\eps$, 
\[
\|u_\eps-\sqrt\eps Z_\epsilon\|_{{\hat T_\eps}/2}\le
 \frac{\eps^{-\frac 14}\,\eps^{\frac 32 -\frac 34}}{|\ln \epsilon|^{3/4}}(\ln|\ln \epsilon|)^6, 
\]
and so,   on $\Gamma_\eps$
\begin{equation}
\nn
\|u_\eps\|_{{\hat T_\eps}/2}\le
\|u_\eps-\sqrt\eps Z_\epsilon\|_{{\hat T_\eps}/2}\,+\,
\|\sqrt\eps Z_\epsilon\|_{{\hat T_\eps}/2}\le
2\,\frac{\epsilon^{\frac{1}{4}}}{|\ln \epsilon|^{1/4}}\,
(\ln|\ln \epsilon|)^2,
\end{equation}
which means that 
 \begin{equation}
 \label{boundUT}
\mathbb P\Big( \| u_\eps\|_{{\hat T_\eps}/2}\le 2\, \frac{\epsilon^{\frac 14}}{|\ln \epsilon|^{1/4}}\,
(\ln|\ln \epsilon|)^2\Big)\to 1 \mbox{ as } \eps \to 0.
\end{equation}
Next, to  consider $u_\eps$ over  the interval $[{\hat T_\eps}/2,\Hat T_\eps]$, notice that $\tau_\eps:={\hat T_\eps}/2= \frac{|\ln\epsilon|}{4}$.  After recalling  that $\upsilon= g\bar u_0+\sqrt\epsilon Z_\eps$,  we obtain 
\begin{equation}
\nn
\|G\varphi(|\upsilon|)\|_{\tau_\eps}\le
    3\, \|G\big(|g\bar u_0 |^3\big)\|_{\tau_\epsilon}
+ 3\,  \|G\big(|\sqrt {\epsilon }Z_\epsilon|^3\big)\|_{\tau_\epsilon}
\end{equation}
Now, if $\displaystyle{\|\bar u_0\|<2\,\frac{ \epsilon^{\frac 14}}{|\ln \epsilon|^{\frac 14}}\,(\ln|\ln\epsilon|)^2}$, 
\[
\|G\big(|g\bar u_0 |^3\big)\|_{\tau_\epsilon}
\le \|\bar u_0\|^3 \int ^{\tau_\epsilon}_0 \e^{\tau_\epsilon-t}\,\e^{3t}\,dt\le 
\|\bar u_0\|^3 \e^{3\tau_\epsilon} 
\le \frac{8\,(\ln|\ln\epsilon|)^6}{|\ln \epsilon|^{\frac 34}}. 
\]
From the last inequality in \eqref{eq:rev1}  with   $T=\tau_\epsilon$, and \eqref{ZTh} with $\rho=2$, we know that 
\[
\mathbb P\big(\,\|G\big(|\sqrt {\epsilon }Z_\epsilon|^3\big)\|_{\tau_\epsilon}< \frac{\epsilon^{\frac 12}(\ln|\ln\epsilon|)^6}
{|\ln \epsilon|^{\frac 34}}\,\big)\to 1 \mbox{ as } \epsilon \to 0. 
\]
Therefore,
\begin{equation}
\label{boundUThat}
\|G\varphi(|\upsilon|)\|_{\tau_\eps}
%\le C \e^{\tau_\eps}\,\big(\|(\sqrt\epsilon Z_\eps)^3\|_{\tau_\eps} +\|\bar u_0^3\|\big)
\le 24\,\Big(\frac{(\ln|\ln \epsilon|)^6}{|\ln \epsilon|^{\frac 34}}\Big)
\end{equation}
in the case $\|\bar u_0\|\le 2\, \frac{\epsilon^{\frac 14}}{|\ln \epsilon|^{1/4}}\,
(\ln|\ln \epsilon|)^2$, with probability tending to $1$.

\begin{lemma} %\label{uTge}
\label{lem:stage1} For each $K>0$ and $\vartheta>0$    
%$\rho \in (0,2K)$, and $\delta \in (0,1)$, there exists $\mathfrak{C} >0$ such that
%$$ 
%\liminf_{\eps \to 0} 
%\bP\big(\, \lambda\,\{r\in [-K,K]: |u_\epsilon(r\sqrt{|\ln \epsilon|},\hat T_\eps)|>
% \frac{\mathfrak{C}}{|\ln \epsilon|^{\frac{1}{4}}} \}>2K-\rho\big) > 1 - \delta.
%$$
$$ 
\liminf_{\eps \to 0} 
\bP\,\Big(\, \inf_{r:|r|\le K, |Y_\eps(r)| > \vartheta |\ln\epsilon|^{-\frac 14}} |u_\epsilon(r\sqrt{|\ln \epsilon|},\hat T_\eps)|>
 \frac{\vartheta}{2|\ln \epsilon|^{\frac{1}{4}}} \Big) = 1.
$$
\end{lemma}
\begin{proof} We first show  that 
\begin{equation}
\label{umz}
\mathbb P \big(\|u_\epsilon-\sqrt\epsilon Z_\eps\|_{\hat T_\epsilon}<\frac{25\,(\ln|\ln\epsilon|)^6}{|\ln \epsilon|^{\frac 34}}\big)\to 1\mbox { as }
\epsilon \to 0.
\end{equation}
From \eqref{eq:rev1}--\eqref{boundUT}
we know that , as $\epsilon \to 0$,
\[
\mathbb P\big(\|u_\eps-\sqrt{\epsilon}Z_\epsilon\|_{\hat T_{\epsilon}/2}<2\,
\frac{\epsilon^{1/2}\,(\ln |\ln \epsilon|)^6}{|\ln\epsilon|^{1/4}} \big)\to 1
\]
and 
\[\mathbb P\big(\|u_\eps\|_{\hat T_{\epsilon}/2}<2\,
\frac{\epsilon^{1/4}\,(\ln |\ln \epsilon|)^2}{|\ln\epsilon|^{1/4}}\,\big) \to 1.
\]
Thus to prove \eqref{umz}  it  suffices to show that 
\[
\mathbb P \,\Big(\sup_{t\in[\hat T_{\epsilon}/2,\hat T_{\epsilon}]}
\|u_\epsilon(\cdot,t)-\sqrt\epsilon Z_\eps(\cdot,t)\|<
\frac{24\,(\ln|\ln\epsilon|)^6}{|\ln \epsilon|^{\frac 34}}\,\Big|\,\|u_\eps\|_{\hat T_{\epsilon}/2}<2\,
\frac{\epsilon^{1/4}\,(\ln |\ln \epsilon|)^2}{|\ln\epsilon|^{1/4}} \Big)
\]
tends to 1 as  $\epsilon \to 0$. 
From the Markov property,  the above probability is bounded from below 
by 
\begin{multline}
\nn
\inf_{\|\bar u_0\|\le 2\,
\frac{\epsilon^{1/4}\,(\ln |\ln \epsilon|)^2}{|\ln\epsilon|^{1/4}} }
\mathbb P\, \Big(\|u_\eps(\cdot,\cdot;\bar u_0)-g \bar u_0(\cdot,\cdot)-\sqrt \epsilon Z_\epsilon (\cdot,\cdot)\|_{\hat T_\epsilon /2}<\,
\frac{24\,(\ln |\ln \epsilon|)^6}{|\ln\epsilon|^{3/4}} \Big)
\\ 
\ge \inf_{\|\bar u_0\|\le 2\,
\frac{\epsilon^{1/4}\,(\ln |\ln \epsilon|)^2}{|\ln\epsilon|^{1/4}} } 
\mathbb P\, \Big(\|G|g\bar u_0-\sqrt \epsilon Z_{\epsilon}\,|^3\|_{\hat T_\epsilon /2} \le \frac{24\,(\ln |\ln \epsilon|)^6}{|\ln\epsilon|^{3/4}}\Big),
\end{multline}
where we have used \eqref{u0z} to obtain the last inequality (recall that $\varphi(u)=u^3$), valid for $\epsilon $ sufficiently small.  
Finally, \eqref{umz} follows from \eqref{boundUThat}. 

Recall then the definition \eqref{ypsilon} of $Y_\eps$ to write 
\begin{equation}
    \label{urge}
u_\epsilon(r\sqrt{|\ln\epsilon|},\hat T_\epsilon)
=u_\epsilon(r\sqrt{|\ln\epsilon|},\hat T_\epsilon)-\sqrt \epsilon Z_\epsilon (r\sqrt{|\ln\epsilon|},\hat T_\epsilon)
 +\,Y_\epsilon(r)
\end{equation}
To conclude the proof just observe that %\comusb{ Final de demo alternativa em azul, a anterior estava em termos da Prop 2.1 que nao mais existe, e tb em termos de $Y$ em vez de $Y_\epsilon$.Tem referencia repetida, deve sumir ao apagar o final da anterior demo que ficou no final }
\[
|u_\epsilon(r\sqrt{|\ln \epsilon|},\hat T_\epsilon)-\sqrt \epsilon Z_\epsilon(r\sqrt{|\ln \epsilon|},\hat T_\epsilon)|<\frac{c}{|\ln \epsilon|^{\frac 14}}
\]
 for any given constant $c$ on a set with probability tending to one with $\epsilon$,  as it follows from \eqref{umz}. 

\end{proof}

%    \begin{remark}
%   It follows form the proof that, for each given $K$, the random set $ \big\{r\in [-K,K]:|u_\epsilon(r\sqrt{|\ln \epsilon|},\Hat T_\eps)|<
% \frac{\mathfrak{C}}{|\ln \epsilon|^{\frac{1}{4}}} \big\} $ is a neighbourhood of $\mathcal N_0(Y)$. 
%    \end{remark}

\section{Stage II}\label{sec:stageII}

%\comue{mudei expressão repetida de $T^b_\eps$ para citação no final}
% $$
% T^b_\eps := \frac12|\ln \eps| + \frac{1}{4} \ln |\ln \eps| + b. 
% $$
If $\phi$ and $\Hat \phi$ are real functions, we write $\phi =_L \Hat \phi$, if $\phi(x) = \Hat \phi(x)$ for all $x\in [-L,L]$, and $\phi \le_L \Hat \phi$, if $\phi(x) \le \Hat \phi(x)$ for all $x\in [-L,L]$.

\medskip

\begin{lemma}\label{lem:stageII}
Denote by $v_\epsilon(x,t;\phi_\epsilon)$ the solution of the deterministic Allen-Cahn  equation with initial condition  $\phi_\epsilon$,
\begin{equation}\label{detAC}
    \begin{cases}
    \frac{\partial v_\eps }{\partial t}=\frac 12   \frac{\partial^2 v_\eps }{\partial x^2}+
    v_\eps- v_\eps^3, \ \textrm{ in } [-L_\eps , L_\eps] \times [0,\infty), 
    \\ 
    v_\eps(-L_\eps,t)=v_\eps(L_\eps,t)=0, \ \forall \, t \in [0,\infty),
    \\ 
    v_\eps(x,0) = \phi_\epsilon,
    \end{cases}
\end{equation}
where $\phi_\epsilon:[-L_\epsilon,L_\epsilon]\to \mathbb{R}$ is a bounded continuous function. 
% i.e.
% \begin{align}\label{detACsol}
% v_\epsilon(x,t;\phi_\epsilon) & = (g\phi_\epsilon) (x,t) - \int_0^t \int_{-L_\eps}^{L_\eps} e^{t-s} h_{L_\eps}(x,y,t-s) v^3_\epsilon (y,s;\phi_\epsilon) dy ds . 
% \end{align} 
Then for every positive constants $M$, $\delta$, $b$, $K$, and $q<1$
$$
\lim_{\epsilon \downarrow 0} \bP \Big[ \sup_{\phi_\epsilon =_{K_\epsilon} u_\epsilon(\cdot,\Hat{T}_\eps) \atop \|\phi_\epsilon\| \le M} 
\sup_{|x| \le q K_\epsilon} 
\big| u_\epsilon(x,T_\eps^b) - v_\epsilon \big(x,\frac 14 \ln |\ln \epsilon| + b; \phi_\epsilon\big) \big| \ge \delta \Big] = 0 ,
$$
where
$K_\epsilon = K \sqrt{|\ln \epsilon|} \ll L_\epsilon$, $T^b_\eps$ is given by \eqref{tempo-final} and
$\phi =_{K_\epsilon} \Hat \phi$ means that $\phi(x) = \Hat \phi(x)$ for all $x\in [-K_\epsilon,K_\epsilon]$.
\end{lemma}

\medskip
\begin{proof}
From \eqref{umz} and \eqref{ZTh} with $\rho=0$, we know that 
\[
\mathbb P \big(\|u_\epsilon\|_{\hat T_\epsilon}<\frac{(\ln|\ln \epsilon|)^2}{|\ln \epsilon|^{1/4}} \big)\to 1
\mbox{ as }\epsilon \to 0.
\]
Let $\hat {\tau}^b_\epsilon=\frac14 \ln|\ln \epsilon|+b$. 
From the Markov property, to prove the lemma it is enough to show that, given 
$M,\delta, b, K$ and $q$ as in the statement, 
\[
\sup_{\|\bar u_0\|<\frac{(\ln|\ln \epsilon|)^2}{|\ln \epsilon|^{1/4}}}
\bP \Big[ \sup_{\phi_\epsilon =_{K_\epsilon} \bar u_0 \atop \|\phi_\epsilon\| \le M} 
\sup_{|x| \le q K_\epsilon} 
\big| u_\epsilon(x,\hat\tau_\eps^b;\bar u_0) - v_\epsilon \big(x,\hat\tau_\eps^b; \phi_\epsilon\big) \big| \ge \delta \Big] 
\]
tends to zero with $\epsilon$. 
For this recall that if $\|\phi_\epsilon\|\le M$, then  
$\|v_\epsilon (\cdot,t;\phi_\epsilon)\|
\le M+1$ for all $t\ge 0$, and that if $\|\bar u_0\|\le 1$, then 
$\mathbb P \big(\|u_\epsilon(\cdot,\cdot;\bar u_0)\|_{\hat\tau_\epsilon^b}>2\big)\to 0$ as $\epsilon \to 0$. (For these a priori estimates, see Propositions 5.1 and 5.2 in \cite {BDP}.) 
Then, from  the integral equations \eqref{ieqconci} and \eqref{ieqv} for $u_\epsilon(\cdot;\bar u_0)$ and $v_\epsilon(\cdot;\phi_\epsilon)$  we have
that, for any $t\le \hat \tau _\eps^b$, 
\begin{multline}
\nn
|u_\epsilon(x,t;\bar u_0 ) - v_\epsilon(x,t; \phi_\epsilon)| \le 
\int_L^L e^{t } h_L(x,y,t) |\bar u_0(y)-\phi_\epsilon(y)|\,dy\\
+\sqrt \epsilon| Z_\epsilon (x,t)|+\int_0^{\hat t} \int_{-L}^L e^{t-s} h_L(x,y,t -s) |v^3_\epsilon(y,s;\phi_\epsilon) - u^3_\epsilon(y, s;\bar u_0)|\,dyds.
\end{multline}
From the above estimates for $u_\eps$ and $v_\eps$, we have  
\begin{multline}
\nn
|u_\epsilon(x,t;\bar u_0 ) - v_\epsilon(x,t; \phi_\epsilon)| 
\\ 
\le \int_{|y|>K_\epsilon} e^{t } h_L(x,y,t) |\bar u_0(y)-\phi_\epsilon(y)|\,dy
+\sqrt \epsilon| Z_\epsilon (x,t)|
\\
+3(M+1)^2\int_0^{t} \int_{-L}^L e^{t -s} h_L(x,y,t-s) |u_\epsilon(y,s;\phi_\epsilon) - v_\epsilon(y, s;\bar u_0)|\,dyds.
\end{multline}
Iterating this inequality and  using the semigroup property we obtain that for   any $t\le \hat \tau _\eps^b$
\begin{align}
\nn
|u_\epsilon(x,t;\bar u_0 ) - v_\epsilon(x,t; \phi_\epsilon)| \le \,&\e^{3(M+1)^2\,t}\int_L^L e^{t } h_L(x,y,t) |\bar u_0(y)-\phi_\epsilon(y)|\,dy
\\ \label{n1}
&+\e^{(3(M+1)^2+1)\,t}\, \sqrt \epsilon \|Z_\eps\|_t.
\end{align}
Observe next that 
\[
 \sup_{|x| \le q K_\epsilon} \int_L^L e^{\hat\tau_\epsilon^b} h_L(x,y,\hat\tau_\epsilon^b) |\bar u_0(y)-\phi_\epsilon(y)|\,dy \le 
  \sup_{|x| \le q K_\epsilon} 
\int_{|y|>K_\epsilon} \!\!e^{\hat\tau_\epsilon^b} h_L(x,y,\hat\tau_\epsilon^b) dy, 
\]
thus,  using \eqref{pineq}, we have that for suitable  positive constants $c$ and $C$, 
$$
 \e^{3\,(M+1)^2\,\hat\tau_\epsilon^b}\sup_{|x| \le q K_\epsilon} 
\int_{|y|>K_\epsilon} e^{\hat\tau_\epsilon^b} h_L(x,y,\hat\tau_\epsilon^b) dy \le C |\ln \eps| e^{- c \sqrt{|\ln \eps|}} , \ \forall \, \eps \in (0,1], 
$$
and from \eqref{esup}, 
\begin{equation}
\nn
\lim_{\eps \to 0} \bP \big( \e^{(3(M+1)^2+1)\,\hat\tau_\epsilon^b}\, \sqrt \epsilon \|Z_\eps\|_{\hat\tau_\epsilon^b} \le C |\ln \eps| \,e^{- c \sqrt{|\ln \eps|}}\, \big) = 1.
\end{equation}
The lemma follows from \eqref{n1} with $t=\hat\tau_\epsilon^b$  and the last two estimates. 
\end{proof}

\begin{lemma} \label{lemma-detAC1}
Let $v_\epsilon(x,t;\phi_\eps)$ be the solution of \eqref{detAC} with initial condition  $\phi_\eps \in \mathbf{C}(-L_\eps,L_\eps)$ such that  $|\phi_\eps|\le 1$ and  
$|\phi_\eps| \in  \big(\frac{\vartheta}{|\ln \eps|^{\frac{1}{4}}},1\big)$ 
%$|\phi_\eps| \in  \big(\frac{1}{2|\ln \eps|^{\frac{1}{4}}},1\big)$ 
in $[-K_\eps,K_\eps]$ for some $\vartheta>0$ and $K_\eps = K\sqrt{|\ln \eps|}$. Then, given $\varrho>0$ there exists a constant $b = b(\varrho)$ such that  for every $q<1$
$$
\limsup_{\eps \to 0} \sup_{|x| \le q K_\eps}  \Big| v_\epsilon \big(x,\frac{1}{4} \ln |\ln \eps| + b; \phi_\eps \big) - \sgn(\phi_\eps(x)) \Big| \le \varrho .
$$
\end{lemma}

\begin{proof} Let $\phi_\eps$ be given as in the statement, and  consider the deterministic Allen-Cahn equation on the whole line 
\begin{equation}\label{detACline}
    \begin{cases}
    \frac{\partial \hat v }{\partial t}=\frac 12   \frac{\partial^2 \hat v }{\partial x^2}+
     \hat v- \hat v^3  , \ \textrm{in } \mathbb{R} \times [0,\infty),
    \\ 
    \hat v (x,0) = \psi_\eps(x) , \ x\in \mathbb{R},
    \end{cases}
\end{equation}
with  initial condition $\psi_\eps$ satisfying 
\[
\psi_\eps=_{K_\eps}\phi_\eps\qquad \mbox{and }\psi_\eps(x)=
\begin{cases}\phi_\eps(-K_\eps) &\mbox{ if } x\le -K_\eps\\
\phi_\eps(K_\eps) &\mbox{ if } x\ge K_\eps.
\end{cases}
\]
Its solution $\hat v(x,t;\psi_\eps)$ satisfies the integral equation
\begin{equation}\label{detAClinesol}
\hat v(x,t;\psi_\eps)  = (g_t \psi_\eps) (x) - \int_0^t \int_{-\infty}^\infty e^{t-s} h(x,y,t-s) \hat v^3_\epsilon (y,s;\psi_\epsilon) dy ds . 
\end{equation} 
For a constant initial condition $a$ such that $|a|\in(0,1)$, we point out that $\hat v(x,t;a)=w(t;a)$, where $w(\cdot;a)$ is the solution of the ordinary equation
\[
\dot w=w-w^3,\qquad w(0)=a ,
\]
which is given explicitly by
\begin{equation}
\label{wt}
w(t;a)=\frac{\sgn (a)}{\sqrt{1+\e^{-2t}(\frac{1}{a^2}-1)}}.
\end{equation}
 By monotonicity properties of the solutions to equation \eqref{detACline},
 we know that, $\forall x\in [-L_\eps,L_\eps] $ and $\eps$ sufficiently small to guarantee that $\vartheta/|\ln \eps|^{1/4} < 1$, 
 \begin{align}
 \nn
 1> \hat v(x,t;\psi_\eps)&\ge w(t;\frac{\vartheta}{2|\ln \eps|^{1/4}})\,>0 \mbox{ if $\psi_\eps>0$},
 \\ \nn 
 -1 < \hat v(x,t;\psi_\eps)&\le w(t;\frac{-\vartheta}{2|\ln \eps|^{1/4}})\,<0 \mbox{ if $\psi_\eps<0$}.
 \end{align}
% \comusb{Agreguei duas  desigualdades no display acima para facilitar a leitura} 
 Therefore, from \eqref{wt}, given $\varrho>0$, there exists $b>0$ such that 
 \begin{equation}
\label{hat-v1}
\limsup_{\eps \to 0} \inf_{|x| \le L_\eps} \Big| \hat v \big(x,\frac{1}{4} \ln |\ln \eps| + b ; \psi_\eps \big) \Big| \ge 1-\frac{\varrho}{2}.
\end{equation}
Moreover, we prove next  that 
\begin{align} \label{hat-v2}
 \limsup_{\eps \to 0} \sup_{|x| \le q K_\eps} \Big|   
\hat v \big(x,& \frac{1}{4} \ln |\ln \eps| + b ; \psi_\eps\big)
\\ \nn
& -v_\epsilon \big(x,\frac{1}{4} \ln |\ln \eps| + b ; \phi_\eps 
\big) \Big| = 0.
\end{align}
Indeed, given a continuous function   $\phi$ satisfying Dirichlet boundary conditions on $[-L,L]$, let us define the reflected  function  
$\tilde{\phi}$ on $(-\infty, \infty)$ as follows:
\[
\tilde\phi(x)=\begin{cases}\quad \phi(x)&\mbox { if } x\in [-L,L]\\
-\phi(2L-x)&\mbox { if } x\in [L,3L]
\end{cases}, \quad \mbox{ and } \tilde\phi \mbox{ is periodic with period $4L$}.
\] 
Then, it can be seen from \eqref{rp} that the solution to the Allen-Cahn equation on the whole line with initial condition $\tilde\phi$ solves   the Allen-Cahn equation \eqref{acdeteq} with Dirichlet boundary conditions and initial condition $\phi$ for $x\in [-L,L]$.  
Therefore, we have that, for any $t\ge 0$ and $x\in[-L_\epsilon,L_\epsilon]$, $v_\eps(x,t;\phi_\eps)=
\hat v(x,t;\tilde \phi_\eps)$ and 
\begin{align}
\label{Dif}
\hat v(x,t;\psi_\eps) -&\hat v(x,t;\tilde \phi_\eps) =
\int_{-\infty}^\infty e^{t} h(x,y,t)\big(\psi_\eps-\tilde\phi_\eps \big) (x)
\\ \nn
-& \int_0^t \int_{-\infty}^\infty e^{t-s} h(x,y,t-s)
\big( \hat v^3 (y,s;\psi_\epsilon)-\hat v^3 (y,s;\tilde \phi_\epsilon)\big)
dy ds .
\end{align}
Denote $D_t(x)=|\hat v(x,t;\psi_\eps) -\hat v(x,t;\tilde\phi_\eps)|$, and observe that, since both 
$\|\psi_\eps\|<1$ and $\|\phi_\eps\|<1$, so are 
$\|\hat v(\cdot,t;\psi_\epsilon)\|<1$ and 
$\|\hat v(\cdot,t;\tilde\phi_\epsilon)\|<1$, and then  
$| \hat v^3 (y,s;\psi_\epsilon)-\hat v^3 (y,s;\tilde \phi_\epsilon) |\le 3 D_t(x)$. 
From \eqref{Dif} we thus obtain 
\begin{equation}
\nn
D_t(x)\le 
\int_{-\infty}^\infty e^{t} h(x,y,t) D_0(y)\,dy
+3\int_0^t \int_{-\infty}^\infty e^{t-s} h(x,y,t-s)
D_s(y)
dy ds 
\end{equation}
By iteration, from the semigroup property,  it follows that for some constant $C>0$ 
\[
D_t(x)\le C\e^{3t}\int_{-\infty}^\infty e^{t} h(x,y,t) D_0(y)\,dy\le \,2 C\e^{4t} \int_{|y|\ge K_\eps }h(x,y,t)dy,
\]
 since $D_0(y)=0$ if $|y|\le K_\eps$. Thus,
\[
\sup_{|x| \le q K_\eps}  D_{\hat \tau^b_\eps}(x)\le k
|\ln \eps|\,\e^{-\frac{A^2 (1-q)^2|\ln \eps|}{\ln |\ln \eps|}},
\]
for some constant $k>0$ and \eqref{hat-v2} follows.
To finish the proof, observe that  $\sgn \phi_\epsilon (x)=
\sgn \psi_\epsilon (x)$ if $|x|\le K_\epsilon$, hence  the statement follows from \eqref{hat-v1} and \eqref{hat-v2}.
\end{proof}

%Observe that the solution to the deterministic A--C equation
% with Neuman BC in $[-\ell,ell]$ (for any $\ell$) 
%\[\frac{\partial u}{\partial t}=\frac 12   \frac{\partial^2 u}{\partial x^2}-  ( u^3- u)\qquad u(x,0)=a\]

\medskip

\noindent \emph{Proof of \rm{(b)} in Theorem \ref{thm:main}.}
\eqref{eq:main} follows from Lemma \ref{lem:stage1}, Lemma \ref{lem:stageII} and Lemma \ref{lemma-detAC1}.
\hfill $\square$

\bigskip

\appendix
\section{Gaussian estimates} \label{appendix}
\begin{proof}[Proof of Lemma \ref{x-y} ]
 Let us start by estimating $E\big(X_\epsilon(r)-Y_\epsilon(r)\big)^2$.
 %\comusb{Ago23: Modifiquei a alineacao das formulas }\comu{A estimativa vale com o supremo em $|r|\le K$ dentro do valor esperado?} \comue{ Nao poderia usar a estimativa de Dudley?}%$$
 %E\big[\big(X_\epsilon(r)-Y_\epsilon(r)%\big)^2\big] 
%$$
Using \eqref{defZ} and \eqref{defz0}, we  write
\begin{align}
  E&\big[\big(X_\epsilon(r)-Y_\epsilon(r)\big)^2\big]
  =\int_0^{\hat T}\e^{-2s}\big[  h\big(  r\sqrt{|\ln\epsilon|}, r\sqrt{|\ln\epsilon|}, 2(\hat T-s)\big) 
  \\ \nn
  &\hskip 4 cm +
 h_L\big(r\sqrt{|\ln\epsilon|}, r\sqrt{|\ln\epsilon|},2(\hat T-s)\big) \big]\,ds
 \\
 \nn
&-2 \int_0^{\hat T}  \e^{-2s}\,\int^L_{-L}\,h_L( r\sqrt{|\ln\epsilon|},y,\hat T-s)h( r\sqrt{|\ln\epsilon|},y,\hat T-s)\,dy\,ds
\end{align}
which, from the pointwise inequality \eqref{pineq} is bounded from above by
\begin{align} \label{ex-y2}
  \int_0^{\hat T}\e^{-2s}\big[h\big( &r\sqrt{|\ln\epsilon|}, r\sqrt{|\ln\epsilon|}, 2 (\hat T-s)\big) \\ \nn
 &
  -h_L\big(r\sqrt{|\ln\epsilon|}, r\sqrt{|\ln\epsilon|},2(\hat T-s)\big)\big]\,ds.  
 \end{align}
 From \eqref{rp}, we obtain 
 \begin{align}
 \nn
& h(r\sqrt{|\ln\epsilon|}, r\sqrt{|\ln\epsilon|},2(\hat T-s))
  -h_L(r\sqrt{|\ln\epsilon|}, r\sqrt{|\ln\epsilon|},2(\hat T-s)) 
  \\ \nn
  =&\frac{1}{\sqrt{4\pi(\hat T-s)}}
  -\frac{1}{\sqrt{4\pi(\hat T-s)}}\, \sum _{j\in \mathbb Z}
\big(\,\e^{-\frac{(4jL)^2}{4(\hat T-s)}}\,-\,
\e^{-\frac{(4jL+2L-2r\sqrt{|\ln\epsilon|})^2}{4(\hat T-s)}}\big)
 \end{align}
which can be rewritten as
\begin{align}
\nn
=& \frac{1}{\sqrt{4\pi(\hat T-s)}}\,\Big(\e^{-\frac{(2r\sqrt{|\ln\epsilon|}-2L)^2}{4(\hat T-s)}}
+\sum _{j\ge 1}\big( \e^{-\frac{(4jL+2L-2r\sqrt{|\ln\epsilon|})^2}{4(\hat T-s)}}-
\e^{-\frac{(4jL)^2}{4(\hat T-s)}}\big)
\\ \nn
&\qquad \qquad +\e^{-\frac{(2r\sqrt{|\ln\epsilon|}+2L)^2}{4(\hat T-s)}}
+\sum _{j\ge 1}\big( \e^{-\frac{(4jL+2r\sqrt{|\ln\epsilon|}+2L)^2}{4(\hat T-s)}}-
\e^{-\frac{(4jL)^2}{4(\hat T-s)}}\big)\,\Big)
\\ \nn
& \hskip 1cm  \le \frac{1}{\sqrt{4\pi(\hat T-s)}}\,\big(\e^{-\frac{(2r\sqrt{|\ln\epsilon|}-2L)^2}{4(\hat T-s)}}
+\e^{-\frac{(2r\sqrt{|\ln\epsilon|}+2L)^2}{4(\hat T-s)}}\,\big)
\\ \nn
 &\hskip 1cm  \le\frac{1}{\sqrt{\pi(\hat T-s)}}\,\e^{-\frac{L^2(1-c)^2}{\hat T-s}}, 
  \end{align}
for any $|r|\le \frac{c\,L}{\sqrt {|\ln\epsilon|}}$. Inserting in the last  integral in
\eqref{ex-y2}, it follows that
\[
E\big[\big(X_\epsilon(r)-Y_\epsilon(r)\big)^2\big] \le \e^{-\frac{L^2(1-c)^2}{\hat T}}
 \int_0^{\hat T}\frac{\e^{-2s}}{\sqrt{\pi(\hat T-s)}}\, ds.  
\]
Then, if $\frac{L^2}{\hat T}\to \infty$ as $\epsilon\to 0$, since the integral above is 
$O(\frac{1}{\sqrt {\hat T}})$ 
\[
\sup_{|r|\le \frac{c\,L}{\sqrt {|\ln\epsilon|}}}\,E\big[\hat T^{\frac 12}(X_\epsilon(r)-Y_\epsilon(r))^2 \big]
\to 0\mbox { as } \epsilon \to 0.
\]
In particular, in the case $L=|\ln\epsilon|$, we obtain %\comu{Não conseguimos a estimativa para a esperança do supremo por algo tipo Burkholder-Davis-Gundy neste contexto?}
\[
\sup_{|r|\le \frac{c\,L}{\sqrt {|\ln\epsilon|}}}\,E\big[\hat T^{\frac 12}(X_\epsilon(r)-Y_\epsilon(r))^2\big]
\le K \eps^{2(1-c)^2}
\]
 Given $h>0$, from the definition of $X_\eps$,  is not difficult to estimate 
 \[
 \bE \big[X_\eps(r+h)-X_\eps(r)\big]^2\le K\,h\sqrt {\hat T}. 
 \]
 Indeed, from \eqref{Xeps}
 \begin{align}
 \label{incX}
 \bE \big [&X_\eps(r+h)-X_\eps(r)\big]^2=\int^{\hat T}_0
 \frac{\e^{-2s}}{\sqrt{2\pi(\hat T-s)}}\,
 \big(1-\e^{-\frac{h^2}{2(1-\frac{s}{\hat T}}}\big)\,ds
 \\ \nn
 &=\sqrt{\hat T}\int^1_0 \frac{\e^{-2u\hat T}}{\sqrt{\pi(1-u)}}\,
 \big(1-\e^{-\frac{h^2}{2(1-u}}\big)\,du 
 \\ \nn
 &
 \le  \sqrt{\hat T} \int^1_0 \frac{1}{\sqrt{\pi(1-u)}} (1\land \frac{h^2}{2(1-u)})
\,du  \le K\,h\sqrt {\hat T} .
 \end{align}
 From \eqref{xdif2}, the corresponding estimate for $Y_\eps$ also holds: 
\[
 \bE \big[Y_\eps(r+h)-Y_\eps(r)\big]^2\le K\,h\sqrt {\hat T}. 
 \]
An application of Dudley's inequality analogous to that used to obtain \eqref{esup} yields that 
\[
\bE \Big[\sup_{|r|\le \frac{b\,L_\eps} {|\ln{\epsilon}|} }\hat T^{\frac 14}|(X_\epsilon(r)-Y_\epsilon(r))|\Big]\le K\,
 T^{\frac 14}\,\big(\ln{(\frac{K|\ln{\epsilon}|}{\eps^{2(1-b)^2}})}\big)^{\frac 12}\eps^{2(1-b)^2} .
  \]
  Then, \eqref{xeps-yeps} follows then from Borell's inequality. 
  \end{proof}
\begin{remark}\label{Xcont} An application of the Laplace method  to the second integral  in \eqref{incX} yields  the  more accurate  estimate
\[
\bE \big [X_\eps(r+h)-X_\eps(r)\big]^2\sim 
\frac{\big(1-\e^{-\frac{h^2}{2}}\big)}{2\sqrt{\pi\hat T}} \mbox{ as } \hat T\to \infty,
\]
used in Proposition \ref{prop:convY} to get that the family 
$\{(8\,\pi |\ln\epsilon|)^{\frac 14} X_\epsilon\colon 0<\eps<1\}$ is tight.
\end{remark}

\bigskip

\paragraph{\bf Acknowledgements:} This project started during a visit of Stella Brassesco to IM-UFRJ in 2020, supported by a CNE grant E-26/202.636/2019 and by Universal CNPq project 421383/2016-0. Stella Brassesco acknodwledges the hospitality and support from the IM-UFRJ during this visit. 

\medskip

\paragraph{\bf Conflict of interest:} The authors have no competing interests to declare that are relevant to the content of  this article.

\medskip

\paragraph{\bf Data Availability:} Data sharing is not applicable to this article as no datasets were generated or analyzed during the process.

\end{document}